\documentclass[12pt]{amsart}
\usepackage{amssymb,amsmath,amsthm, geometry, accents, enumerate}
\usepackage[shortlabels]{enumitem}
\usepackage{mathrsfs}
\usepackage{ color, soul, hyperref,mathtools}
\usepackage{stmaryrd} 
\usepackage[all]{xy}
\usepackage{tikz-cd} 
\hypersetup{
    colorlinks=true,
    linkcolor=blue,
    filecolor=magenta,      
    urlcolor=cyan,
    citecolor=blue,
}
\everymath{\displaystyle}

\newcommand*\xbar[1]{%
  \hbox{%
    \vbox{%
      \hrule height 0.5pt 
      \kern0.5ex
      \hbox{%
        \kern-0.1em
        \ensuremath{#1}%
        \kern-0.1em
      }%
    }%
  }%
}
\def\temp{&} \catcode`&=\active \let&=\temp

\newtheorem{thm}{Theorem}[section]
\newtheorem{cor}[thm]{Corollary}
\newtheorem{defn}[thm]{Definition}
\newtheorem{lemma}[thm]{Lemma}
\newtheorem*{lemma*}{Lemma}
\newtheorem*{thm*}{Theorem}
\newtheorem{prop}[thm]{Proposition}
\theoremstyle{definition}
\newtheorem{eg}[thm]{Example}
\newtheorem{defn-rmk}[thm]{Definition-Remark}
\theoremstyle{remark}
\newtheorem{rmk}[thm]{Remark}
\newtheorem{notation}[thm]{Notation}
\newtheorem{question}[thm]{Question}

\newcommand{\Exterior}{\mathchoice{{\textstyle\bigwedge}}%
    {{\bigwedge}}%
    {{\textstyle\wedge}}%
    {{\scriptstyle\wedge}}}

\DeclareMathOperator{\wid}{Width}
\DeclareMathOperator{\width}{W}
\DeclareMathOperator{\Tor}{Tor}
\DeclareMathOperator{\Ann}{Ann}
\DeclareMathOperator{\Ass}{Ass}
\DeclareMathOperator{\ch}{char}
\DeclareMathOperator{\grade}{grade}
\DeclareMathOperator{\Supp}{Supp}
\DeclareMathOperator{\pd}{pd}
\DeclareMathOperator{\End}{End}
\DeclareMathOperator{\Ext}{Ext}
\DeclareMathOperator{\Hom}{Hom}
\DeclareMathOperator{\supph}{Suppi}
\renewcommand{\min}{\mbox{\rm min}}
\DeclareMathOperator{\im}{Im}
\DeclareMathOperator{\modr}{mod(\textit{R})}

\DeclareMathOperator{\tildes}{\underaccent{\tilde}{\it{s}}}
\DeclareMathOperator{\tildef}{\underaccent{\tilde}{\it{f}}}
\DeclareMathOperator{\tildei}{\underaccent{\tilde}{\it{i}}}
\DeclareMathOperator{\tildej}{\underaccent{\tilde}{\it{j}}}
\newcommand{\ei}{\mbox{$e_{\underaccent{\tilde}{i}}$}}
\newcommand{\ej}{\mbox{$e_{\underaccent{\tilde}{j}}$}}

\renewcommand{\P}{\mbox{$\mathcal{P}_{m}^{d}$}}
\newcommand{\Pm}{\mbox{$\mathcal{P}_{m-1}^{d}$}}
\newcommand{\Pmm}{\mbox{$\mathcal{P}_{m+1}^{d}$}}

\newcommand{\CA}{\mbox{$\mathcal{A}$}}
\newcommand{\CB}{\mbox{$\mathcal{B}$}}

\newcommand{\T}{\mbox{$\mathcal{T}$}}
\newcommand{\CP}{\mbox{$\mathcal{P}$}}

\newcommand{\cp}{\mbox{$\mathfrak{p}$}}
\newcommand{\m}{\mbox{$\mathfrak{m}$}}
\renewcommand{\L}{\mbox{$\mathscr{L}$}}

\newcommand{\DbLP}{\mbox{${D^b_{\tiny \L}}(\CP)$}}
\newcommand{\ChbLT}{\mbox{${Ch^b_{\tiny \L}}(\CT)$}}
\newcommand{\DbLT}{\mbox{${D^b_{\tiny \L}}(\T)$}}

\newcommand{\BN}{\mbox{$\mathbb N$}}\newcommand{\BZ}{\mbox{$\mathbb Z$}}
\newcommand{\CT}{\mbox{$\mathcal{T}$}}

\begin{document}

\pagenumbering{arabic}

\title{Derived equivalences via Tate resolutions}
\date{}
\author{K. Ganapathy and Sarang Sane}

\address{K. Ganapathy\\
Department of Mathematics \\
Indian Institute of Technology Madras\\
Sardar Patel Road \\
Chennai India 600036}
\email{ganapathy.math@gmail.com}
\address{Sarang Sane*\\
Department of Mathematics \\
Indian Institute of Technology Madras\\
Sardar Patel Road \\
Chennai India 600036}
\email{sarangsanemath@gmail.com}

\begin{abstract}
  For any finite sequence of elements $s_1, \ldots , s_d$ in a commutative noetherian ring $R$, we show that for $n \gg 0$, the natural map from the Koszul complex $K(s_1^n, \ldots , s_d^n)$ to the Koszul complex $K(s_1, \ldots , s_d)$ factors through the Tate resolution on $s_1^n, \ldots , s_d^n$. Using this, for any resolving subcategory $\CA$ of $\modr$ and any ideal $I$ such that it has a filtration $\{ I_n \}$ which is equivalent to the $I$-adic filtration and $\textrm{dim}_{\tiny \CA}(R/I_n) < \infty$, we show a derived equivalence between the bounded derived category of finitely generated modules supported on $V(I)$ having finite $\CA$-dimension and the bounded derived category of $\CA$ with homologies supported on $V(I)$. As a special case, when $R$ is of prime characteristic and $I$ is of finite projective dimension, we obtain a derived equivalence between the bounded derived category of finite projective dimension modules supported on $V(I)$ and the bounded derived category of projective modules with homologies supported on $V(I)$.
\end{abstract}

\maketitle

\section{Introduction}
Let $R$ denote a commutative noetherian ring and $\modr$ the category of finitely generated $R$-modules. Let $\CP$ denote the full subcategory of projective $R$-modules and $\CA$ denote a resolving subcategory of $\modr$. Let $\xbar \CA$ denote the full subcategory of $\modr$ consisting of modules which have a finite resolution by modules in $\CA$, i.e., $M \in \modr$  such that $\text{dim}_{\tiny \CA}(M) < \infty$. This is equivalent to demanding that all higher syzygies of $M$ lie in $\CA$. When $\CA = \CP$,  $\text{dim}_{\tiny \CA}$ is just the projective dimension and we abbreviate it by $\pd$. For a subcategory $\mathcal{K} \subset \modr$, let $Ch^b(\mathcal{K})$ denote the category of bounded chain complexes in $\mathcal{K}$. For a Serre subcategory $\L$ of $\modr$, let $Ch^{b}_{\tiny \L}(\mathcal{K})$ be the full subcategory of $Ch^b(\mathcal{K})$ consisting of complexes with homologies in $\L$ and $D^b_{\tiny \L}(\mathcal{K})$ denote the corresponding derived category obtained by inverting quasi-isomorphisms. Associating a complex to its resolution by $\CA$-modules yields
a functor $D^b(\xbar \CA \cap \L) \rightsquigarrow D^b_{\tiny \L}(\CA)$. A Serre subcategory $\L$ of $\modr$ is said to satisfy condition (*) if given an ideal $I \subseteq R$, such that $R/I \in \L$, there exists a regular sequence $a_1,\dots,a_c$ in $I$ such that $R/(a_1,\dots,a_c) \in \L$ (\cite[Definition 2.11]{SandersSane}). When $\L$ satisfies condition (*) (for example, when $\L = \L_{V(I)}$ where $I$ is a set-theoretic complete intersection), it was proved in \cite[Lemma 4.6]{SandersSane}, that the above functor is an equivalence. 
In this article, we generalize the above-mentioned result to a much larger class of Serre subcategories (refer Theorem~\ref{derivedequivalence}). In particular, we prove the following:
\begin{thm}\label{intro-der-equiv}
Let $R$ be a commutative noetherian ring. Let $\CA \subseteq \modr$ be a resolving subcategory. Then there is an equivalence of categories $ D^b(\xbar{\CA} \cap \L) \rightsquigarrow {D^b_{\tiny \L}}(\CA)$ in the following cases :
    \begin{enumerate}
        \item $R$ is a regular ring.
        \item $\L$ satisfies condition $(*)$.
        \item $\L = \L_{V(I)}$ where $I \subseteq R$ is an ideal which has a filtration $\{ I_n \}$ which is equivalent to the $I$-adic filtration and $\textrm{dim}_{\tiny \CA}(R/I_n) < \infty$.
    \end{enumerate}
\end{thm}
Theorem \ref{intro-der-equiv} consolidates both the known cases of the derived equivalence, namely (1) and (2), and extends it in the third case. We define an ideal $I$ to have eventually finite projective dimension (efpd) (refer Definition \ref{def-efpd}) if it satisfies the property in (3) in Theorem \ref{intro-der-equiv} when all $I_n$ have finite projective dimension. An immediate and particularly interesting consequence of Theorem \ref{intro-der-equiv} is :
\begin{cor}\label{intro-cor}
\begin{enumerate}[(a)]
    \item Let $\L = \L_{V(I)}$ where $I$ has efpd. Then $ D^b( \xbar{\CA} \cap \L ) \simeq {D^b_{\tiny \L}}(\CA) $.
    \item If $R$ is of prime characteristic and $\pd(R/I) < \infty$, then $I$ has efpd and hence the above derived equivalence holds.
    \item When $R$ is a regular ring or $\L$ satisfies condition {\rm(*)} or $\L = \L_{V(I)}$ where $I$ has efpd, then $ D^b( \xbar{\CP} \cap \L ) \simeq \DbLP $.
\end{enumerate}
\end{cor}
The efpd property appears in  literature before (without a name), in particular in \cite{Peskine_Szpiro_thesis}, where the more general case of filtrations of modules with this property is considered. We formally define ideals with efpd in subsection \ref{efpd-subsection} of the preliminaries and discuss some examples and related questions after the definition and in Section \ref{last section}.

The main technique in the proof of Theorem \ref{intro-der-equiv} is to show that for any complex $X_{\bullet}$ in ${D^b_{\tiny \L}}(\CA)$ with more than one non-zero homology, there exists a morphism from a suitable projective resolution in ${D^b_{\tiny \L}}(\CA)$ such that the cone has one homology less. This is an example of what we define as a strong reducer for $X_{\bullet}$ (refer definition \ref{SR prop}). This idea has been used earlier in \cite{SandersSane}, where Koszul complexes were used to produce strong reducers. In fact, once strong reducers exist, the proofs in \cite{SandersSane} almost verbatim produce the derived equivalence. These proofs along with the proof of Theorem \ref{intro-der-equiv} occupy much of Section \ref{section der-equiv}.

In this article, we show that suitable Tate resolutions play the role of strong reducers. The procedure to do this is via morphisms from Tate resolutions to Koszul complexes. We describe the relevant result next since it is of independent interest.
Let $\tildes = s_1,\dots,s_d $ be any sequence of elements in $R$ and $K(\tildes)$ the Koszul complex on $\tildes$. Recall that the Tate resolution $T (\tildes)$ is obtained by resolving successive homologies starting from $K(\tildes)$ to form a projective resolution of $R/(\tildes)$. In fact, both complexes are differential graded algebras, and by construction, there is a natural inclusion $K(\tildes) \rightarrow T (\tildes)$ of dg-algebras.
It is well-known that $K(\tildes)$ and $T (\tildes)$ ``coincide" if and only if $\tildes$ is a regular sequence. In other words, an inverse $\xymatrix{ T (\tildes) \ar@{-->}[r] & K(\tildes) \\}$ exists if and only if $\tildes$ is a regular sequence. Now let $\tildes^n$ denote the sequence obtained from $\tildes$ by taking $n^{th}$ powers of each of the elements of $\tildes$. The natural surjection $R/(\tildes^{u}) \twoheadrightarrow R/(\tildes^r)$ induces a natural chain complex map $\kappa^{u,r} : K(\tildes^{u}) \rightarrow K(\tildes^r)$ which is the identity map in the case $u = r$. In Theorem \ref{chainmapexists}, we prove that for \textit{any} sequence $\tildes$, for sufficiently large $u$, the map $\kappa^{u,r}$ factors through $T(\tildes^u)$, i.e.,
\begin{equation}\label{intro-diag}
\xymatrix{ K(\tildes^u) \ar@{^{(}->}[r] \ar[rd]_{\kappa^{u,r}}& T(\tildes^u) \ar@{-->}[d]^{\phi} \\ & K(\tildes^r) \\} .
\end{equation}
When $\tildes = \{s\}$, a singleton, this phenomenon and the role of raising powers and $R$ being noetherian can be seen as follows : Since $R$ is noetherian, there exists $l \in \BN$ such that $\Ann(s^{l}) = \Ann(s^{l+1})$. The diagram below shows the map $\phi$ when $n = l+1$.
$$\xymatrix{ T (s^{{l}+1}): \quad \ar[r] & R^{t} \ar[r]^-{\partial} \ar[d]^{0} & Re \ar[r]^{(s^{{l}+1})} \ar[d]^{(s^{l})}  & R \ar[r] \ar[d]^{id}& 0 \\ K(s): \quad  \ar[r]  & 0 \ar[r] & Re \ar[r]^{(s)}  & R \ar[r]  & 0}.$$
Since $\partial$ maps a basis of $R^t$ to the generators of $\Ann(s^{l+1}) = \Ann(s^{l})$, we see that $\phi$ is a chain complex map.

Let $N$ be an $R$-module. Another interesting consequence of diagram \eqref{intro-diag} is that it shows a deeper connection between Koszul cohomologies $\left\{ H^i(K(\tildes^r;N)^\vee) \right\}_{r \geq 1}$ and $\Ext$-modules $\left\{ \Ext^i_R(R/(\tildes^{r}),N) \right\}_{r \geq 1}$, as shown in Corollary \ref{connection tor ext}. As an immediate consequence, we get a new proof that the local cohomology module $H^i_I(N)$ is isomorphic to the direct limit $\lim_{\longrightarrow} H^i(K(\tildes^r;N)^\vee)$ without using the \v{C}ech complex.

 The proof of the existence of $\phi$ relies on vanishing statements for natural maps on Koszul homology modules after raising powers, which are known to follow from the statement that $\Tor^R_i(R/I^n,M) \to \Tor^R_i(R/I^r,M)$ is zero for $n \gg 0$, which is an implication of the 
Artin-Rees lemma. We show further that the existence of $\phi$ is really equivalent to the vanishing statement above for the Tor maps. The exponent $n$ in the statement, thought of as a function of $r$ and $i$ has been studied before, for example, in \cite{andre_michel_commalg}. In Theorem \ref{uniform artin rees}, we show that $n$ can be chosen to be
a polynomial function on $r$ and is also independent of the degree $i$, which is an immediate consequence of Theorem \ref{chainmapexists}. This vanishing statement for the Tor maps has been studied for local rings in \cite{Eisenbud_Huneke} and \cite{Aberbach} (also refer Definition \ref{SAR defn} and its commentary).

The connection between the Artin-Rees lemma and the derived equivalence is classical. In \cite{Grothendieck}, there are conditions under which $D^b(\CB) \rightarrow D^b_{\tiny \CB}(\CA)$ is an equivalence, where $\CA$ is an abelian category and $\CB$ is a Serre subcategory. In \cite[Pg. 17, Example (b)]{Keller}, it is checked that the conditions hold when $\CA = \modr$ and $\CB$ is the full subcategory of $I$-torsion modules for an ideal $I \subseteq R$. The proof crucially uses the Artin-Rees lemma.

A final consequence of the existence of $\phi$ is to obtain bounds on two K-theoretic invariants $\alpha(\L)$ and $\beta(\L)$ which were introduced in \cite{ChandaSane} (refer Section~\ref{sec-prelims}) for more details). Using Theorem \ref{chainmapexists}, it can be shown that $\beta(\L)$ and hence $\alpha(\L)$ can be bounded by the projective dimension of $R/I$ when $R$ is of prime characteristic, $\L = \L_{V(I)}$ and $I$ is a perfect ideal.

\section{Preliminaries} \label{sec-prelims}
Let us fix some notations: Let $d,n \in \BN$ and 
$\P$ be the set of all ordered $m$-subsets of $[d] = \{ 1, \ldots , d \}$. Let $s_1,\dots,s_d \in R$. Denote $\tildes^n = \{ s_1^n,\dots,s_d^n \}$  and $ \tildes = \tildes^1$.
Recall that the Koszul complex $K(\tildes^n) = \Exterior(\oplus_{i=1}^d Re_i)$ is a differential graded algebra with $\deg(e_i) = 1$ for all $i$ and the differential mapping $e_i$ to $s_i^n$. Hence the algebra structure on $K(\tildes^n)$ is the exterior algebra for all $n \in \BN$.

Let $M \in \modr$. Recall that the Koszul complex on $\tildes^n$ with coefficients in $M$ is $
K(\tildes^n;M) = \left( K_m{(\tildes^n;M)}, \partial_m^{K(\tildes^n;M)}\right)_{m = 0}^{d}$ where $K_m{(\tildes^n;M)} = K_m{(\tildes^n)} \otimes_R M$ and $\partial_m^{K(\tildes^n;M)}$ is the induced map. A typical element of $K_m{(\tildes^n;M)} = \bigoplus_{\tildei \in \tiny \P} M\ei$ is of the form $\sum_{\tildei \in \tiny \P} c_{\tildei} \ei$ where $c_{\tildei} \in M$. Some authors denote $H_m(K(\tildes^n; M))$ by $H_m(\tildes^n; M)$. \\
     In the introduction, a basic example of how chains of annihilator ideals are useful was demonstrated. In a similar spirit, the modules 
     $(0:_M s^t)$ will play an important role in the article. We state some basic properties of these submodules.
\begin{notation}\label{notations1}{\rm
    Since $R$ is noetherian and $M$ is a finitely generated $R$-module, for each $s \in R$, the ascending chain of submodules of $M$, \[ (0:_M s) \subseteq \cdots \subseteq (0:_M s^t) \subseteq (0:_M s^{t+1}) \subseteq \cdots \] stabilizes. Let $l_i \in \BN$ be such that $(0:_M s_i^{{l}_i}) = (0:_M s_i^{{l}_i+k})$ for all $k \in \BN$. We define  $$l = l(\tildes) \coloneqq \max \{l_i \mid i \in [d] \}. $$  
}\end{notation}
    \begin{lemma}\label{annihilatorlemma}
        \begin{enumerate}[(i)]
            \item Let $k \geq 0$. Then $\left( 0:_M \left( \prod_{t \in [d]}s_t^{l+k}\right) \right) = \left(0:_M \left(\prod_{t \in [d]}s_t^{l}\right) \right)$.
            \item Let $x,y \in M$ such that $s_i^{l}$ divides $x$ and $y$. If $s_i x = s_i y$, then $x = y$. 
        \end{enumerate}
    \end{lemma}
    \begin{proof}
        (i) Let $m \in M$ such that $\left( \prod_{t \in [d]}s_t^{l+k}\right) m = 0$. We have $(0:_M s_i^{{l}_i}) = (0:_M s_i^{{l}_i+k})$ for all $i \in [d]$. Hence $s_d^{l}\left( \prod_{t \in [d-1]}s_t^{l+k}\right) m  =0$. Continuing inductively, we get $\left(\prod_{t \in [d]}s_t^{l}\right) m = 0$. 
        Hence, the LHS is contained in the RHS. The other containment is clear, and this proves (i).
        \par (ii) Since $s_i^l$ divides $x$ and $y$, $x-y = s_i^l z$ for some $z \in M$. Given that $s_i (x -  y) = 0$. Hence $s_i^{l+1}z = 0$. By the definition of $l$, $s_i^lz = 0$. Therefore $x-y =0$ which implies $x = y$.
    \end{proof}

\subsection*{The Artin-Rees Lemma and vanishing of maps on Tor modules}
\label{ar-subsection}
\begin{lemma}[Artin-Rees]
    Let $R$ be a commutative noetherian ring and $I$ an ideal in $R$. Let $M$ be a finitely generated $R$-module and $N \subseteq M$. Then there exists $h \in \BN$ such that for all $r \geq 0$,
    $$I^{r+h}M \cap N = I^{r}(I^hM \cap N).$$
\end{lemma}
\noindent A weaker version of the Artin-Rees lemma which is often useful is 
\begin{equation*}\label{weaker artin rees}
    I^{r+h}M \cap N \subseteq I^{r} N.
\end{equation*}

\begin{lemma}\label{Lem-Syz-Tor}
    Let $I\subseteq R$ be an ideal, $P_\bullet$ a projective resolution of  $M$ and $B_i \subseteq P_i$ the image of $\partial_{i}^P$. Then 
    $$\Tor_i^R(M,R/I) = \Tor_1^R(P_{i-1}/B_{i-1},R/I) = (I P_{i-1} \cap B_{i-1})/I B_{i-1}.$$
\end{lemma}
\begin{proof}
    From  \cite[Exercise 2.4.3]{weibelhomological}, $\Tor_i^R(M,-) = \Tor_1^R(Z_{i-2},-)$ and since $P_\bullet$ is exact, $$Z_{i-2} = B_{i-2} = P_{i-1}/ Z_{i-1} = P_{i-1}/ B_{i-1}.$$ This proves the first equality. Applying the long exact sequence of $\Tor$ to the exact sequence $0 \to B_{i-1} \to P_{i-1} \to P_{i-1}/B_{i-1} \to 0$ and noting that $ \Tor_1^R(P_{i-1},R/I) = 0$ yields
    \begin{align*}
        \Tor_1^R(P_{i-1}/B_{i-1},R/I) & = \ker \left( B_{i-1} \otimes R/I \to P_{i-1} \otimes R/I \right) \\
        & = \ker \left( B_{i-1}/IB_{i-1} \to P_{i-1}/IP_{i-1} \right) \\
        & = (I P_{i-1} \cap B_{i-1})/I B_{i-1}.
    \end{align*}
\end{proof}

\begin{rmk}
     By the Artin-Rees lemma, for all $i\geq 1$, there exists $h(i) \in \BN$ such that for all $r \geq 0$,
    \begin{equation}\label{artin rees rsln}
        I^{r+h(i)} P_{i-1} \cap B_{i-1} \subseteq I^r B_{i-1}.
    \end{equation}
    By Lemma \ref{Lem-Syz-Tor}, the natural map   
    $$(I^{r+h(i)} P_{i-1} \cap B_{i-1})/I^{r+h(i)} B_{i-1} \to (I^r P_{i-1} \cap B_{i-1})/I^r B_{i-1}$$ is same as the natural map $$\Tor_i^R(R/I^{r+h(i)},M) \to \Tor_i^R(R/I^r,M).$$
    The above map on $\Tor_i$ is zero if and only if $I^{r+h(i)} P_{i-1} \cap B_{i-1} \subseteq I^r B_{i-1}$, which is true from equation \eqref{artin rees rsln}. This also shows that equation \eqref{artin rees rsln} is independent of the projective resolution.
\end{rmk}

As an immediate consequence, we get the following result.
\begin{prop} \label{andre-homology}
Let $I\subseteq R$ be an ideal and $M$ a finitely generated $R$-module. For all $i \geq 1$, there exists $h(i) \in \BN$ such that for all $r \geq 1$, the map 
        \( \Tor^R_i(R/I^{r+h(i)},M) \to \Tor^R_i(R/I^r,M) \) is zero.
\end{prop}
\noindent In \cite{Eisenbud_Huneke} and \cite{Aberbach}, 
equation \eqref{artin rees rsln} is studied for free resolutions of the module $M$.
\begin{defn} \cite{Aberbach} \label{SAR defn}
    Let $R$ be a commutative noetherian ring and $M$ a finitely generated $R$-module. Then $M$ is said to be syzygetically Artin-Rees with respect to an ideal $I$ if there exists a free resolution $P_\bullet$ and a uniform integer $h$  such that for all $r \geq 0$ and for $i \geq 0$,
    \begin{equation}\label{SAR}
    I^{r+h} P_i \cap B_i \subseteq I^{r} B_i.
\end{equation} 
\end{defn}
 In \cite[Corollary 4.9]{Aberbach}, it is proved that any finitely generated module $M$ is syzygetically Artin-Rees with respect to any ideal $I$ in a local ring which answers Question B in \cite{Eisenbud_Huneke}. A similar vanishing statement for maps between $\Tor$ modules has been studied in \cite{andre_michel_commalg}, which uses a weaker version of the Artin-Rees lemma.
 
 As mentioned in the introduction, we will crucially make use of the vanishing of maps between Tor modules for most part of Section \ref{section main thm} and will use Definition \ref{SAR defn} and its commentary for Theorem \ref{uniform artin rees} and its corollary.

\subsection*{The intersection theorem and Cohen-Macaulay rings}
In this subsection, we recall two well-known results for later use.
The following result is due to Peskine-Szpiro and Roberts.
\begin{thm}[Intersection theorem] 
    Let $R$ be a local ring and $M$ and $N$ non-zero finitely generated modules over $R$. If $M \otimes_R N$ has finite length, then $\dim(N) \leq \pd(M).$
\end{thm}
An immediate well-known consequence is the following criterion for a noetherian local ring to be Cohen-Macaulay.
\begin{cor}\label{intersection cor}
    Let $R$ be a noetherian local ring. Then there exists a non-zero finitely generated $R$-module of finite length and finite projective dimension if and only if $R$ is Cohen-Macaulay. 
\end{cor}

\subsection*{Filtrations of ideals and asymptotic behaviour of projective dimension}\label{efpd-subsection}
\begin{defn}
    A filtration $\{J_n\}$ of $I$, i.e., $I  \supseteq J_1 \supseteq \cdots  \supseteq J_{n-1} \supseteq J_n \supseteq \cdots$ where each $J_n$ is an ideal in $R$, is said to be equivalent to the $I$-adic filtration $\{I^k\}$ if for all $k \in \BN$, there exists $n$ such that $J_n \subseteq I^k$ and for all $n \in \BN$, there exists $k$ such that $I^k \subseteq J_n$.
\end{defn}
Equivalently, the filtration $\{J_n\}$ induces the $I$-adic topology.
\begin{eg}\label{equivalent filtration eg}
    \begin{enumerate}[(i)]
        \item For an ideal $I$ in $R$ with a generating set $\tildes$, we define $I^{[r]}_{\tildes} := (\tildes^r)$. Clearly $I^{[r]}_{\tildes} \subseteq I^r$ and $I^{rd} \subseteq I^{[r]}_{\tildes}$ where $d$ denotes the number of elements in the sequence $\tildes$. Hence, the square power filtration $\{ I^{[n]}_{\tildes} \}$ is equivalent to the power filtration $\{I^k\}$.
        \item For a regular ring $R$ and an ideal $I \subseteq R$, the symbolic power filtration $\{I^{(n)}\}$ of $I$ is known to be equivalent to $\{I^k\}$ (\cite[Theorem 3.5]{JKV}). In fact, the hypothesis on $R$ can be further weakened (\cite[Proposition 4.7]{JKV}).
    \end{enumerate}
\end{eg}
\begin{rmk}
    Let $I$ be an ideal of $R$ and $\tildes$ a generating set of $I$. If $\ch(R) = p > 0$ is prime, then 
    $I^{[p^n]}_{\tildes}$ is the $n$-th Frobenius power of $I$ and hence is independent of the generating set $\tildes$. We denote it by $I^{[p^n]}$.
\end{rmk}

The next lemma follows directly from \cite[Theorem 1.7]{Peskine_Szpiro_thesis}.
\begin{lemma} \label{char p proj dim}
    Let $I\subseteq R$ be an ideal and $\ch(R) = p >0$, where $p$ is a prime. If $\pd(R/I) < \infty$ then $\pd(R/I) = \pd(R/I^{[p^n]})$ for every $n \in \BN$.
\end{lemma}

\begin{defn} \label{def-efpd}
Let $I \subseteq R$ be an ideal.
    The ideal $I$ is said to have eventually finite projective dimension (efpd) if there is a filtration $\{J_n\}$ of $I$ equivalent to $\{I^k\}$ such that $\pd(R/J_n) < \infty$ for all $n \in \BN$. 
\end{defn}

\begin{rmk}\label{propertiesefpd}
    If $I, J \subseteq R$ are ideals such that $\sqrt{I} = \sqrt{J}$, then $I$ has efpd if and only if $J$ has efpd. 
\end{rmk}

We discuss a few classes of ideals which have eventually finite projective dimension.
\begin{eg}\label{efpd example}
    \begin{enumerate}
        \item When $R$ is a regular ring, every ideal has finite projective dimension and hence every ideal has efpd. 
        \item Let $\ch(R) = p >0$, where $p$ is a prime and $\pd(R/I)< \infty$. Then by Lemma \ref{char p proj dim}, $\pd(R/I) = \pd(R/I^{[p^n]})$ for all $n \in \BN$. Hence, it follows from Example \ref{equivalent filtration eg}(ii) that $I$ has efpd.
        \item \label{CI eg} Let $I$ be a set-theoretic complete intersection ideal, i.e., there exists a regular sequence $a_1,\dots,a_n$ such that $\sqrt{(a_1,\dots,a_n)} = \sqrt{I}$. Since $a_1^k,\dots,a_n^k$ is also a regular sequence for each $k \in \BN$, the Koszul complex $K(a_1^k,\dots,a_n^k)$ is a free resolution of $R/(a_1^k,\dots,a_n^k)$. Thus $\pd(R/(a_1^k,\dots,a_n^k)) < \infty$ for all $k \in \BN$. Hence every set-theoretic complete intersection ideal has efpd.
        \item As a special case of (\ref{CI eg}), the maximal ideal in a Cohen-Macaulay local ring has efpd. 
        \item Every nilpotent ideal has efpd, since the filtration $\{I^n\}$ eventually becomes zero.
        \item Let $R$ be an artinian ring with $Maxspec(R) =  \{\m_1,\dots,\m_t\}$ for some $t \geq 1$. Then there exists $n \in \BN$ such that $R \cong R/\m_{1}^{n} \times \cdots \times R/\m_{t}^{n}$ and  $(\m_{1} \cdots \m_{t})^{n} = 0$. We claim that every ideal $I \subseteq R$ has efpd. Without loss of generality, we may assume $I$ is a radical ideal. Hence $I = \prod_{i \in \Lambda} \m_{i}$ for some subset $\Lambda \subseteq [t]$. If $R$ is a local ring, then $I$ is the maximal ideal which is nilpotent. Hence $I$ has efpd. Suppose there are more than one maximal ideal in $R$. Then we may assume $\Lambda \subset [t]$ to be any non-empty proper subset. Then for all $k \geq n$, $$I^k = \prod_{i \in \Lambda} \m_{i}^{k} \cong \prod_{j \in [t]\backslash \Lambda} R/\m_{j}^{k} \cong \prod_{j \in [t]\backslash \Lambda} R/\m_{j}^{n}$$ is a projective $R$-module. Hence $I$ has efpd. 
        \item Let $f: R \to S$ be a flat ring homomorphism and $I \subseteq R$ has efpd. Then there exists a filtration $\{J_n\}$ of $I$ such that $\pd_R(R/J_n) < \infty$ and $\{J_n\}$ is equivalent to $\{I^k\}$. Hence $\pd_S(S/(J_nS ))<\infty$ and $\{J_nS\}$ is equivalent to $\{{(IS)}^k\}$. Thus $IS \subseteq S$ has efpd. 
        \item \label{efpd localizes} If $I \subseteq R$ has efpd, then $I_{\cp} \subseteq R_{\cp}$ has efpd for all $\cp \in Spec(R)$, since localisation is flat.
        \item\label{Dutta Hochster Mclaughlin eg} 
        Let $R = k[x_1,x_2,x_3,x_4]_{\tiny \m}/(x_1x_4 - x_2x_3 )$ where $k = \BZ/p\BZ$ and $\m = (x_1,x_2,x_3,x_4)$. The construction of Dutta-Hochster-McLaughlin\cite{Dutta_Hochster_Mclaughlin} shows the existence of an $R$-module $M$ of finite length and finite projective dimension such that the intersection multiplicity $\chi_R(M,N) < 0$ for some module $N$. The existence of a flat local ring homomorphism from a regular local ring to $R$ would imply that $\chi_R(M,N) \geq 0$. Therefore there is no flat local ring homomorphism from any regular local ring to $R$. Since $R$ has a module of finite length and finite projective dimension, Corollary \ref{intersection cor} implies that $R$ is a Cohen-Macaulay local ring. Thus $\m R \subseteq R$ has efpd. Thus we have an example of an ideal having efpd, which is not extended from a regular local ring under a flat map.
        \end{enumerate}
\end{eg}
\indent Later in Section 4, we show that a maximal ideal in a noetherian local ring has efpd if and only if the ring is Cohen-Macaulay. We also give a necessary condition for arbitrary ideals to have eventually finite projective dimension. 
\begin{eg}\label{efpd non-example}
    Let $R = k[X,Y]/(XY)$. The ideal $(X)$ in $R$ does not have efpd.
\end{eg}

\subsection*{Subcategories of mod(R)}
We denote by $\modr$ the category of finitely generated $R$-modules.

\begin{defn}\label{serresub defn}
\begin{enumerate}
    \item A Serre subcategory of $\modr$, denoted by $\L$, is a full subcategory of $\modr$ such that for any short exact sequence $0 \to M' \to M \to M'' \to 0$  in $\modr$, $M \in \L$ if and only if $M', M'' \in \L$.
    \item A set $V \subseteq \text{Spec}(R)$ is called a specialization closed set if it is a union of closed sets.
    \item Given a specialization closed set $V$, there is a corresponding Serre subcategory defined as $$\L_V = \{ M \in \modr \mid \\Supp(M) \subseteq V \}.$$
    \item \label{VL defn} Given a Serre subcategory $\L$, there is a corresponding specialization closed set defined as
    $$V_{\tiny \L} = \bigcup_{M \in {\tiny \L}} \Supp(M).$$
\end{enumerate}
    
\end{defn}
\begin{rmk}\label{serresub}\cite{Gabriel}
Clearly, $\L \subseteq \L_{V_{\tiny \L}}$. Suppose $M \in \L_{V_{\tiny \L}}$. Then, by definition, $$\Ass(M) \subseteq \Supp(M) \subseteq V_{\tiny \L} = \bigcup_{M' \in {\tiny \L}} \Supp(M').$$ Thus, for each $\cp \in \Ass(M)$, there exists a module $M' \in \L$ such that $\cp \in \Supp(M')$. Since $\L$ is a Serre subcategory, it follows that $R/{\cp} \in \L$, since it is a quotient of some associated prime of $M'$, which is in $\L$ since it is a submodule of $M'$. Using the existence of a finite filtration for $M$ whose quotients are $R/{\cp}$ where $\cp \in \Ass(M)$, we obtain that $M \in \L$ and hence that $\L = \L_{V_{\tiny \L}} $.
\end{rmk}

\begin{defn}
    A full subcategory $\CA \subseteq \modr$ is said to be resolving if 
    \begin{enumerate}
        \item $R$ is in $\CA$.
        \item $M \oplus N \in \CA$ if and only if $M$ and $N$ are in $\CA$. 
        \item If $0 \to M' \to M \to M''\to 0$ is an exact sequence and $M'' \in \CA$, then $M \in \CA$ if and only if $M' \in \CA$.
    \end{enumerate}
\end{defn}
As a consequence, the category of finitely generated projective $R$-modules $\CP$ is contained in any resolving subcategory. Some examples of resolving subcategories are $\CP, \modr$, the category of Gorenstein dimension zero modules (also called totally reflexive modules) and the category of maximal Cohen-Macaulay (MCM) modules in a Cohen-Macaulay local ring. Basic properties and statements about resolving subcategories may be found in \cite{Dao-Takahashi} and \cite{Sanders}. The collection of resolving subcategories of $\modr$ is classified in these articles, and in particular, it shows that $\modr$ has many resolving subcategories.

\begin{defn}
    A full subcategory $\T \subseteq \modr$ is a thick subcategory of $\modr$ if it is resolving and for any exact sequence in $\modr$, $$0 \to M' \to M \to M'' \to 0,$$  $M,M' \in \T$ implies $M'' \in \T$.
\end{defn}

\begin{rmk}
    Let $\Omega^i M$ denote the $i^{th}$ syzygy of $M$.
    For a resolving subcategory $\CA$, define the closure $\overline{\CA} = \{ M \in \modr \mid \Omega^{\gg 0} M \in \CA \}$.
    Then $\overline{\CA}$ is a thick subcategory of $\modr$.
\end{rmk}

\begin{eg} Here are some examples of thick subcategories of $\modr$ which arise as a closure of a resolving subcategory. 
    \begin{enumerate}
        \item $\xbar{\CP}$ is the category of finitely generated modules having finite projective dimension.
        \item The closure of Gorenstein dimension zero modules is the full subcategory of $\modr$ having finite Gorenstein dimension.
        \item When $R$ is a Cohen-Macaulay local ring, $\overline{MCM(R)} = \modr$.
    \end{enumerate}
\end{eg}

We fix the following notations:
    For $X_\bullet \in Ch^b(\modr)$,
    \begin{enumerate}
        \item $\min_{c}(X_{\bullet}) = \sup\{ n \mid P_{\bullet} = 0 \text{ for all } i <n  \}$
        \item $\min(X_\bullet) = \sup\{ n \mid H_i(X_{\bullet}) = 0 \text{ for all } i <n  \}$ 
        \item $\supph(X_{\bullet}) = \{ n \mid H_n(X_{\bullet}) \neq 0 \}$
        \item $\wid(X_\bullet) = \begin{cases}
            \sup\{ i-j \mid H_i(X_\bullet), H_j(X_\bullet) \neq 0 \} & \text{if\ } X_\bullet \text{\ is not acyclic} \\ 
            -\infty & \text{if\ } X_\bullet \text{\ is acyclic} 
        \end{cases} $
        \item $\width(X_\bullet) = \begin{cases}
            0 & \text{if\ } \wid(X_\bullet) = -\infty \\
            \wid(X_\bullet) & \text{otherwise}
        \end{cases}$
        \item $\Sigma X_\bullet$ denotes the shift of the complex $X_\bullet$, that is, $(\Sigma X_\bullet)_k = (X_\bullet)_{k+1}$ 
    \end{enumerate}

\subsection*{Preliminaries on \texorpdfstring{$K_0$}{K0} of certain categories}

In \cite{ChandaSane}, the stable range $\alpha(\L)$ of $K_0$ relative to a Serre subcategory $\L$ of $\modr$ and a computable invariant $\beta(\L) \geq \alpha(\L)$ were introduced and studied. 
We recall below the definitions and results about $\beta(\L)$ which are relevant to this article. Later in Section 3, we give bounds on $\beta(\L)$ in certain cases.
\begin{defn}
\begin{enumerate}
    \item Let $k \geq 1$ and $P_\bullet \in Ch^{[0,k]}_{\tiny \L}(\CP)$. A pair $(Q_\bullet,u)$ is said to be a reducer of $P_\bullet$ if $Q_\bullet \in Ch^{[0,k-1]}_{\tiny \L}(\CP)$ and $u: Q_\bullet \to P_\bullet$ is a chain complex map such that $u_0: Q_0 \to P_0$ is surjective.
    \item $\beta(\L) = \inf\{ m \in \BN \mid \text{every}\ P_\bullet \in Ch^{[0,k]}_{\tiny \L}(\CP) \ \text{has a reducer, for all}\ k>m \}$.
\end{enumerate}
\end{defn}
Recall the notation that $\tildes = s_1,\dots,s_d$. A careful observation and reformulation in terms of the dual of the chain complex of \cite[Proposition 23]{Foxby-Halvorsen} yields the next lemma.
\begin{lemma}\label{Foxby-Halvorsen}
    Let $I = (\tildes)$ be an ideal and $\L = \L_{V(I)}$. For $P_\bullet \in Ch^{-}_{\tiny \L}(\CP)$, there exists $r \in \BN$ and a chain complex map $\Psi : K(\tildes^r ; P_0) \to P_\bullet$ such that $\Psi_0$ is an isomorphism.
\end{lemma}
Lemma \ref{Foxby-Halvorsen} is crucially used in this form in \cite{SandersSane} and later in \cite{ChandaSane} to obtain the first example of a reducer and which also yields statement (2) of the next lemma.
\begin{lemma}\label{beta properties} \cite{ChandaSane}
Let $I = (\tildes)$ be an ideal and $\L = \L_{V(I)}$. 
    \begin{enumerate}
        \item When $\L \neq 0$, $\grade(I) \leq \beta(\L)$.
        \item $\beta(\L) \leq d$.
    \end{enumerate}
\end{lemma}
A direct consequence of \cite[Theorem 6.3]{ChandaSane} is stated below.
\begin{lemma}\label{grade-beta equality}
    If $\grade(I) = \beta(\L)$ then $K_0(\xbar \CP \cap \L) \cong K_0(\DbLP)$.
\end{lemma}

\begin{rmk}
We recall that $\beta(\L_{V(I)})$ is bounded above by $\limsup \left\{ \pd(R/I^{[r]}_{\tildes}) \mid n \in \BN \right\}$ as proved in \cite[Theorem 5.1]{ChandaSane}. We note that the same proof shows that the bound can be obtained with any filtration $\{J_n\}$ equivalent to  $\{I^k\}$, i.e.,
    $$\beta(\L_{V(I)}) \leq \limsup \left\{ \pd(R/J_n) \mid n \in \BN \right\} .$$
    Note that the RHS of the above inequality is also an upper bound for the non-vanishing of the local cohomology $H^i_I(M)$ for an $R$-module $M$, since it is the direct limit of $\Ext^i_R(R/J_n,M)$.
\end{rmk}

\subsection*{Notations for further use}
\label{notations}
We collect below some notations that will be used in the rest of the article.
\begin{itemize}
    \item $\tildei = \{ i_1,\dots,i_m \} \in \P$ with $i_j < i_{j+1}$. \\
    \item $\ei$ := $e_{i_1} \wedge  \dots \wedge e_{i_m}$. \\
    \item $P_{\tildei}$ : the $R$-linear map $\bigoplus_{\tildei \in \tiny \P} M\ei \to M$ where $\sum_{\tildei \in \tiny{\P}} c_{\tildei} \ei$ maps to $c_{\tildei}$.\\
    \item $\tildei \sqcup k$ : the $(m+1)$-ordered subset containing $\tildei$ and $k$ where $k \notin \tildei$. \\
    \item $\kappa^{n,k}$ : the chain complex map from $K(\tildes^{n};M)$ to $K(\tildes^k;M)$ where $\kappa^{n,k}_m$ maps $\ei$ to $(s_{i_1}\cdots s_{i_m})^{n-k} \ei$ for all $\tildei \in \P$ and $m \in [d]$ where $n \geq k >0$.\\
    \item $l$ : the common stabilization point for the chain of the submodules $(0:_M s_i^t)$ (defined in notations~\ref{notations1}).
\end{itemize}
Note that for $n \geq t \geq k$, $\kappa^{n,k} = \kappa^{t,k} \circ \kappa^{n,t}$. When $M=R$, $\kappa^{n,k}$ is a map of dg algebras.

\section{Factoring $\kappa$ through the Tate resolution}\label{section main thm}
In this section, we define the Tate resolution and prove the main theorem about the existence of a map from a suitable Tate resolution to the Koszul complex as mentioned in the introduction. We begin with a lemma about the vanishing of maps on Koszul homologies.
\begin{lemma}\label{homology map is zero}
    There exists $h \in \BN$ such that for all $r \geq 1$, the map  $$H_i(\kappa^{h+rd,r}) : H_i(K(\tildes^{h+rd};M)) \longrightarrow H_i(K(\tildes^r;M))$$ is zero for all $i \geq 1$. For $i=0$, the map is the natural surjection $\displaystyle{M/({\tildes}^{h+rd})M \to M/({\tildes^r})M}$.
\end{lemma}
\begin{proof}
    Let $S = R[X_1,\dots, X_d]$ be the polynomial ring and $f: S \to R$ the ring map mapping $X_i$ to $s_i$. We denote the $S$-module structure on $M$ through $f$ by $\prescript{}{S}{M}$.   Then for all $h \geq 0$, we have the following commutative diagram : 
    $$\xymatrix{  \Tor_i^S(S/(X_1^{h+rd},\dots,X_d^{h+rd}),\prescript{}{S}{M}) \ar[r]^{\cong} \ar[d]^{} & H_i(K(X_1^{h+rd},\dots,X_d^{h+rd};\prescript{}{S}{M})) \ar[r]^-{\cong} \ar[d]^{}  & H_i(K(\tildes^{h+rd};M)) \ar[d]^{} \\  \Tor_i^S(S/(X_1^r,\dots,X_d^r),\prescript{}{S}{M}) \ar[r]^{\cong} & H_i(K(X_1^r,\dots,X_d^r;\prescript{}{S}{M})) \ar[r]^-\cong  & H_i(K(\tildes^r;M))   }$$
    Since the Koszul complex is concentrated in degrees zero to $d$, it is sufficient to prove that for all $1 \leq i \leq d$, there exists $h \geq 0$ such that
    \begin{equation}\label{tor-map}
    \Tor_i^S(S/(X_1^{h+rd},\dots,X_d^{h+rd}),\prescript{}{S}{M}) \to \Tor_i^S(S/(X_1^r,\dots,X_d^r),\prescript{}{S}{M})
    \end{equation}
     is the zero map. Since the natural surjection $S/(X_1^{h+rd},\dots,X_d^{h+rd}) \twoheadrightarrow S/(X_1^r,\dots,X_d^r)$ factors through $S/(X_1,\dots,X_d)^{h+rd} \twoheadrightarrow S/(X_1,\dots,X_d)^{rd}$, the map in \eqref{tor-map} is zero if the natural map $\Tor_i^S(S/(X_1,\dots,X_d)^{h+rd},\prescript{}{S}{M}) \to \Tor_i^S(S/(X_1,\dots,X_d)^{rd},\prescript{}{S}{M})$ is zero. Using Proposition \ref{andre-homology} with the exponent $rd$, there exists $h(i) \in \BN$, for each $1 \leq i \leq d$, such that for all $r \geq 1$, the map 
\begin{equation*}
    \Tor_i^S(S/(X_1,\dots,X_d)^{h(i)+rd},\prescript{}{S}{M}) \to \Tor_i^S(S/(X_1,\dots,X_d)^{rd},\prescript{}{S}{M})
\end{equation*}
 is zero. Choosing $h = \max\{h(i) \mid 1 \leq i \leq d \}$, we get that
 \begin{equation*}
    \Tor_i^S(S/(X_1,\dots,X_d)^{h+rd},\prescript{}{S}{M}) \to \Tor_i^S(S/(X_1,\dots,X_d)^{rd},\prescript{}{S}{M})
\end{equation*}
 is the zero map for all $1 \leq i \leq d$. 
\end{proof}
We fix the integer $h$ (which depends on $M$ and $\tildes$) obtained above for the rest of this section, i.e., for all $r \geq 1$,
    \begin{equation*}
    H_i(\kappa^{h+rd,r}): H_i(K(\tildes^{h+rd};M)) \longrightarrow H_i(K(\tildes^r;M))
\end{equation*}
 is the zero map for all $1 \leq i \leq d$. The following technical lemma about the divisibility of elements is used for our induction statement in Theorem \ref{chainmapexists}. 

\begin{lemma}\label{T2}
 Given $n \geq 1$, $r \geq 1$, there exists $v := q(n,r) = h+(n+r)d$ such that for all  $m \geq 1$ and $x \in \ker\left(\partial_m^{K({\tildes}^v ; M)}\right)$, there exists $y \in K_{m+1}(\tildes^r ; M)$ such that :
 \begin{enumerate}[(a)]
     \item $\partial_{m+1}^{K(\tildes^r;M)}(y) = \kappa^{v,r}_m(x)$.
     \item $P_{\tildej}(y)$ is divisible by $\prod_{j \in \tildej} s_j^{n}$ for all $\tildej \in \Pmm$.
 \end{enumerate} 
\end{lemma}   
\begin{proof}
From Lemma \ref{homology map is zero}, the natural map $H_m(\kappa^{v,n+r}) : H_m(K(\tildes^v ; M)) \rightarrow H_m(K(\tildes^{n+r} ; M))$ is zero for all $m \geq 1$. Therefore for all $x \in \ker\left(\partial_m^{K({\tildes}^v ; M)}\right)$, we get $\kappa^{v,n+r}_m(x) \in B_m(K(\tildes^{n+r} ; M))$, where $B_m(K(\tildes^{n+r} ; M))$ denotes $\im(\partial_{m+1}^{K({\tildes^{n+r} ; M)}})$. Hence the commutative diagram below shows that $\kappa^{v,r}_m(x) \in B_m(K(\tildes^r ; M))$.
  $$ \begin{tikzcd}[sep=large]  
 & K_{m+1}(\tildes^{n+r};M) \arrow[twoheadrightarrow]{d}{\partial_{m+1}^{K(\tildes^{n+r};M)}} \arrow{r}{\kappa^{n+r,r}_{m+1}} & K_{m+1}(\tildes^{r};M) \arrow[twoheadrightarrow]{d}{\partial_{m+1}^{K(\tildes^r;M)}} \arrow{r}{P_{\tildei}} & \left( \prod_{i \in \tildei} s_i^{n-1} \right) M  \\
 \ker\left( \partial_m^{K(\tildes^v;M)} \right) \arrow[r,"\kappa^{v,n+r}_m","(Lemma\ \ref{homology map is zero})"']  &  B_m(K(\tildes^{n+r};M))  \arrow{r}{\kappa^{n+r,r}_m} &  B_m(K(\tildes^r;M)) &
\end{tikzcd}$$ 
Since $\kappa^{v,n+r}_m(x) \in B_m(K(\tildes^{n+r};M))$, there exists $z = \sum_{\tildej \in \tiny{\Pmm}} c_{\tildej} \ej \in K_{m+1}(\tildes^{n+r};M)$ such that $\partial_{m+1}^{K(\tildes^{n+r};M)}(z) = \kappa^{v,n+r}_m(x)$. Let $y: = \kappa^{n+r,r}_{m+1}(z)$. Then $\partial_{m+1}^{K(\tildes^r;M)}(y) = \kappa^{v,r}_m(x)$. Also, 
$$ y = \sum_{\tildej \in \tiny{\Pmm}} \left( \prod_{j \in \tildej} s_j^{n} \right) c_{\tildej} \ej \quad \textrm{and hence} \quad 
P_{\tildej}(y) = \left( \prod_{j \in \tildej} s_j^{n} \right)  c_{\tildej} \quad \textrm{ for all } \tildej \in \Pmm.$$
\end{proof}
 
\begin{thm}\label{T1}
  Let $r \geq 1$, $m \geq 1$ and $x = \sum_{\tildei \in \tiny{\P}} x_{\tildei} \ei \in \ker\left( \partial_{m}^{K(\tildes^r;M)} \right)$. Suppose $x_{\tildei}$ is divisible by $\prod_{i \in \tildei}s_i^{l+n}$  for each $\tildei \in \P$. Then there exists $y\in \ker\left( \partial_m^{K(\tildes^{n+r};M)} \right) $ such that $x = \kappa^{n+r,r}_m(y)$.
\end{thm}
\begin{proof}
We first consider the case $m = 1$. Then $x =\sum_{i =1}^{d}x_i e_i$ and $\partial_1^{K(\tildes^r;M)}(x) = \sum_{i =1}^{d} s_i^r x_i  = 0$. Since $s_i^{l+n}$ divides $x_i$, there exists $c_i \in M$ such that $x_i = s_i^{l+n} c_i$ for all $i \in [d]$. Define $y =\sum_{i =1}^{d}s_i^{l} c_i  e_i \in K_1(\tildes^{n+r};M)$. Then $\partial_1^{K(\tildes^{n+r};M)}(y) = \sum_{i =1}^{d} s_i^{l+n+r} c_i = \sum_{i =1}^{d} s_i^r x_i = 0$. Further, 
$\kappa^{n+r,r}_1(y) =\sum_{i =1}^{d} s_i^{l+n} c_i  e_i= x$, thus proving the statement in this case.\par Consider the case $m \geq 2$. Since $x = \sum_{\tildei \in \tiny \P} x_{\tildei} \ei  \in \ker\left(\partial_{m}^{K(\tildes^r;M)}\right)$, we have
\begin{equation}
    \sum_{\tildej \in \tiny \Pm} \left( \sum_{k \in [d]\backslash \tildej} (-1)^{\tau(k,\tildej)+1} s_k^r x_{(\tildej \sqcup k)} \right) \ej = 0
\end{equation}
where $\tau(k,\tildej)$ denotes the position of $k$ in the ordered set $\tildej \sqcup k$. Hence for each $\tildej \in \Pm$, 
\begin{equation}\label{K1}
    \sum_{k \in [d]\backslash \tildej} (-1)^{\tau(k,\tildej)+1} s_k^r x_{(\tildej \sqcup k)} = 0.
\end{equation}
Since $\prod_{t \in \tildei}s_t^{l+n}$ divides $x_{\tildei}$, we get $x_{\tildei}= \left( \prod_{t \in \tildei}s_t^{l+n} \right) x'_{\tildei}$ for some $x'_{\tildei} \in M$. Thus for each $\tildej \in \Pm$ and $k \in [d] \setminus \tildej$, we obtain
\begin{align*}
    0 = \sum_{k \in [d]\backslash \tildej} (-1)^{\tau(k,\tildej)+1} s_k^r x_{(\tildej \sqcup k)} & = \sum_{k \in [d]\backslash \tildej} (-1)^{\tau(k,\tildej)+1} s_k^r \left( \prod_{t \in (\tildej \sqcup k)}s_t^{l+n}\right) x'_{(\tildej \sqcup k)} \\
    & = \left(\prod_{t \in \tildej}s_t^{l}\right)\left(\prod_{t \in \tildej}s_t^{n} \right)  \sum_{k \in [d]\backslash \tildej} (-1)^{\tau(k,\tildej)+1} s_k^{n+r} (s_k^{l}  x'_{(\tildej \sqcup k)}) .\\
\end{align*}
By Lemma \ref{annihilatorlemma}, for all $n \geq 1$, $\left(0:_M \left( \prod_{t \in \tildej}s_t^{n+l}\right) \right) = \left( 0:_M\left(\prod_{t \in \tildej}s_t^{l}\right)\right)$. Thus,
$$\left( \prod_{t \in \tildej}s_t^{l} \right) \sum_{k \in [d]\backslash \tildej} (-1)^{\tau(k,\tildej)+1} s_k^{n+r} (s_k^{l}  x'_{(\tildej \sqcup k)}) = 0.$$ 
 Take $y_{\tildei} = \left( \prod_{t \in \tildei}s_t^{l} \right) x'_{\tildei} $. Observe that $x_{\tildei} = \left( \prod_{t \in \tildei}s_t^n \right)y_{\tildei}$. Define $y := \sum_{\tildei \in \tiny{\P}} y_{\tildei} \ei \in K_m{(\tildes^{n+r};M)}$. Then, for all $\tildej \in \Pm$, the above equation yields
$$ \sum_{k \in [d]\backslash \tildej} (-1)^{\tau(k,\tildej)+1} s_k^{n+r}   y_{(\tildej \sqcup k) } = 0 .$$ 
Therefore $y \in \ker\left( \partial_m^{K(\tildes^{n+r};M)}\right)$. Clearly $\kappa^{n+r,r}_m(y) = x$. 
\end{proof} 

\begin{defn}\label{defn-rmk}
Let $\tildef = f_1 \ldots , f_n \in R$. Taking intuition from the construction of the Tate resolution (refer for example \cite{Avramov}), we now construct a chain complex $T'(\tildef;M)$. Define
\[
T'_i(\tildef;M) = \begin{cases}
            0 & \textrm{ if }~  i < 0 \\
            K_i(\tildef;M) & \textrm{ if }~ i = 0, 1 \\
            K_i(\tildef;M) \oplus R^{t_i} & \textrm{ if }~  i \geq 2
        \end{cases}
\]
where the differentials and $t_i$ are inductively defined as follows : \\$\partial_1^{T'(\tildef;M)} := \partial_1^{K(\tildef;M)}$. Choose a generating set $S_1$ of $Z_1(T'(\tildef;M))$ and define $t_2 = \lvert S_1 \rvert$. Suppose for all $j \leq r-1$, the differentials $\partial_j^{T'(\tildef;M)}$ have been defined , generating sets $S_j$ of $Z_j(T'(\tildef;M))$ have been chosen and $t_{j+1} = \lvert S_j \rvert$. Define $\partial_r^{T'(\tildef;M)}$ by
$\partial_r^{T'(\tildef;M)}|_{K_r(\tildef;M)} = \partial_r^{K(\tildef;M)}$ and $\partial_r^{T'(\tildef;M)}$ maps a basis of $R^{t_r}$ to the set $S_{r-1} $ of $Z_{r-1}(T'(\tildef;M))$. Choose a generating set $S_r$ of $Z_r(T'(\tildef;M))$ and define $t_{r+1}=\lvert S_r \rvert$.
\end{defn}

Note that $T'(\tildef;M)$ satisfies the following properties : 
 \begin{enumerate}
        \item $K(\tildef;M)$ is a subcomplex of $T'(\tildef;M)$.
        \item $H_i(T'(\tildef;M)) = 0$ for all $i \geq 1$ and $H_0(T'(\tildef;M)) = M/(\tildef)M$.
    \end{enumerate}
 A slightly enhanced version of the above construction when $M = R$ in order to obtain a dg $R$-algebra structure is precisely the Tate resolution on $R/(\tildef)$ as defined in \cite{Avramov}, which we denote by $T (\tildef)$. We are now set up to state and prove the main theorem.

\begin{thm}\label{chainmapexists}
 Given $r \geq 1$, there exists  a chain complex map $\phi: T'(\tildes^{u(r)};M) \to K(\tildes^r;M)$ with $\phi|_{K(\tildes^{u(r)};M)} = \kappa^{u(r),r}$ for $u(r) = (h+l)\left(\sum_{j = 0}^{d-1}d^j\right) + rd^{d}$. In particular, there exists a chain complex map $\phi: T (\tildes^{u(r)}) \to K(\tildes^r)$ such that $\phi_0 = \kappa^{u(r),r}_0 = id_R$.
 \end{thm}
\begin{proof}
 Recall the notation $q(n,r) = h+(n+r)d$ from Lemma \ref{T2}. Define
 $$
 q^{(i)} = 
 \begin{cases}
     q(0,r)-r = h + rd -r & \quad \textrm{ if }~ i = 0 \\
     q(q^{(i-1)}+ l,r)-r & \quad \textrm{ if }~ i \geq 1 \\
 \end{cases} \qquad \qquad . $$
Claim : For $i\geq 0$, $q^{(i)} = h\left(\sum_{j = 0}^{i}d^j\right) + rd^{i+1} + l \left(\sum_{j = 0}^{i-1}d^j\right) d -r$. \\
 The claim is clearly true for the base case $i=0$ and follows by induction as shown by the calculation below : 
 \begin{align*}
     q^{(i+1)} = q(q^{(i)}+l,r)-r  & = h + \left(h\left(\sum_{j = 0}^{i}d^j\right) + rd^{i+1} + l \left(\sum_{j = 0}^{i-1}d^j\right) d -r+ l + r \right) d -r \\
     & = h\left(\sum_{j = 0}^{i+1}d^j\right) + rd^{i+2} + l \left(\sum_{j = 0}^{i}d^j\right) d -r.
 \end{align*}
  Therefore $q^{(d-1)}+l+r = h\left(\sum_{j = 0}^{d-1}d^j\right) + rd^d + l \left(\sum_{j = 0}^{d-2}d^j\right) d + l = (h+l)\left(\sum_{j = 0}^{d-1}d^j\right) + rd^{d} = u(r)$. For the rest of this proof, we abbreviate $u(r)$ by $u$. We now proceed to construct $\phi$. For $m = 0$, define $\phi_0$ to be the identity map. For $m > d$, the zero map is the only possible map.
  For $1 \leq m \leq d$, we will inductively define $\phi_m$ such that\begin{enumerate}[(a)]
      \item $\phi_{m-1} \partial_m^{ T'(\tildes^{u};M)} = \partial_m^{K(\tildes^r;M)} \phi_m$
      \item $\phi_m|_{K_m(\tildes^u;M)} = \kappa^{u,r}_m$
      \item $P_{\tildei} \phi_m$ is divisible by $\prod_{i \in \tildei} s_i^{q^{(d-m)}+l}$ for all $\tildei \in \P$.
  \end{enumerate} 
   Then, we will show that the constructed map $P_{\tildei}\phi_d$ being divisible by $\prod_{i \in \tildei} s_i^{q^{(0)}+l}$ will imply that $\phi_{d} \partial_{d+1}^{\tiny T'(\tildes^{u};M)} = 0 = \partial_{d+1}^{K(\tildes^r;M)} \phi_{d+1}$, which will complete the proof. \\
  \indent In the base case $m=1$, note that $T'_1(\tildes^{u};M) = M^d \cong \bigoplus_{i = 1}^{d} Me_i$ and so we can define $\phi_1$ by extending $\phi_1(e_i) = s_i^{u-r} e_i $. Then $\phi_{0} \partial_1^{ T'(\tildes^{u};M)} = \partial_1^{K(\tildes^r;M)} \phi_1$, $\phi_1|_{K_1(\tildes^u;M)} = \kappa^{u,r}_1$ and $P_i \phi_1$ is divisible by $ s_i^{q^{(d-1)}+l}$, i.e., conditions (a), (b) and (c) hold.
  \par Suppose $\phi_{t}$ exists for $1 \leq t<m \leq d$ such that the conditions (a), (b) and (c) hold. Recall that $T'_m(\tildes^u;M) = K_m(\tildes^u;M) \oplus R^{t_m}$. We will define $\phi_m$ on each of the summands separately and check that (a), (b) and (c) above are satisfied on the summands. Define $\phi_m|_{K_m(\tildes^u;M)} := \kappa^{u,r}_m$. Since $\partial_m^{ T'(\tildes^{u};M)}|_{K_m(\tildes^u;M)} = \partial_m^{ K(\tildes^{u};M)}$ by definition and $\kappa^{u,r}$ is a chain complex map, $\phi_{m-1} \partial_m^{ T'(\tildes^{u};M)}|_{K_m(\tildes^u;M)} = \partial_m^{K(\tildes^r;M)} \phi_m|_{K_m(\tildes^u;M)}$. Since $P_{\tildei}\kappa^{u,r}_m$ is divisible by $\prod_{i \in \tildei} s_i^{u-r}$ for all $\tildei \in \P$ and $u-r \geq q^{(d-m)}+l$, the condition (c) follows for the summand $K_m(\tildes^u;M)$. Let $x$ be a basis element in the free $R$-module $R^{t_m} \subseteq T'_m(\tildes^{u};M)$. By induction, $\phi_{m-1} \partial_m^{ T'(\tildes^{u};M)}(x) \in \ker\left(\partial_{m-1}^{K(\tildes^r;M)}\right)$ and $P_{\tildei}\phi_{m-1} \partial_m^{\tiny T'(\tildes^{u};M)}(x)$ is divisible by $\prod_{i \in \tildei} s_i^{q^{(d-m+1)}+l}$. From Theorem \ref{T1}, there exists $z \in \ker\left( \partial_{m-1}^{K(\tildes^{q^{(d-m+1)}+r};M)}\right)$ such that $\kappa^{q^{(d-m+1)}+r,r}_{m-1}(z) = \phi_{m-1} \partial_m^{ T'(\tildes^{u};M)}(x)$. Since $q^{(d-m+1)}+r = q(q^{(d-m)}+ l,r)$ and $z \in \ker\left( \partial_{m-1}^{K(\tildes^{q^{(d-m+1)}+r};M)}\right)$, Lemma \ref{T2} implies that there exists $y \in K_{m}(\tildes^r;M)$ such that $\partial_{m}^{K(\tildes^r;M)}(y) = \kappa^{q^{(d-m+1)}+r,r}_{m-1}(z)$ and $P_{\tildei} (y)$ is divisible by $\prod_{i \in \tildei} s_i^{q^{(d-m)}+l}$. Define $\phi_m(x) := y$. Then, $$\phi_{m-1} \partial_m^{ T'(\tildes^{u};M)}(x) = \kappa^{q^{(d-m+1)}+r,r}_{m-1}(z) = \partial_{m}^{K(\tildes^r;M)}(y)=   \partial_m^{K(\tildes^r;M)} \phi_m(x).$$ 
  Extend this to a map $\phi_m$ on 
  $R^{t_m}$ by linearity. 
  Since for every basis element of $R^{t_m}$, the properties (a) and (c) are satisfied, both of these hold for every element in $R^{t_m}$. Thus $\phi_m$ has been constructed inductively for $1 \leq m \leq d$. \\
  \indent It remains to check that $\phi_{d} \partial_{d+1}^{\tiny T'(\tildes^{u};M)} = 0$. Let $x \in T'_{d+1}(\tildes^{u};M)$. Since $P_{\tildei}\phi_{d} \partial_{d+1}^{\tiny T'(\tildes^{u};M)}(x)$ is divisible by $\prod_{i \in \tildei} s_i^{q^{(0)}+l}$, by Theorem \ref{T1}, there exists $z \in \ker\left(\partial_d^{K(\tildes^{q^{(0)}+r};M)}\right)$ such that $\kappa^{q^{(0)}+r,r}_{d}(z) = \phi_{d} \partial_{d+1}^{\tiny T'(\tildes^{u};M)}(x)$. Since $q^{(0)}+r = q(0,r)$ and $z \in \ker\left(\partial_d^{K(\tildes^{q^{(0)}+r};M)}\right)$, Lemma \ref{T2}(a) implies that $\kappa^{q^{(0)}+r,r}_{d}(z) \in \im\left(\partial_{d+1}^{K(\tildes^r;M)}\right) = 0$.  This completes the proof. 
  \end{proof}
  \begin{rmk}
      Note that in the previous Theorem~\ref{chainmapexists}, the obtained $\phi$ need not be unique since at each degree, there could be several choices of lifts. Since both the Tate resolution $T (\tildes^{u(r)})$ and the Koszul complex $K(\tildes^r)$ are dg $R$-algebras, one may ask if we could construct the above map such that it is a dg $R$-algebra map. Begin with any map $\phi: T (\tildes^{u(r)}) \to K(\tildes^r)$ obtained from Theorem~\ref{chainmapexists}.
      Since the dg $R$-algebra $T (\tildes^{u(r)})$ is a polynomial algebra of the form $R[\{X_{ij} \mid i >0, j \in \Lambda_i \}]$ where $X_{ij}$ are variables of degree $i$, we can define a dg $R$-algebra map $\gamma: T (\tildes^{u(r)}) \to K(\tildes^r)$ by mapping $\gamma(X_{ij}) = \phi(X_{ij})$ and then extend it to all of $T (\tildes^{u(r)})$ using the dg-structure. Hence there exists a map of dg $R$-algebras from $T (\tildes^{u(r)})$ to $K(\tildes^r)$ which extends $\kappa^{u(r),r}$.
  \end{rmk}
  \begin{cor}\label{chain map for modules}
      Let $r \geq 1$ and $u(r)$ denote the integer as in Theorem \ref{chainmapexists}. Let $P_\bullet(M/(\tildes^{u(r)})M)$ be a projective resolution of $M/(\tildes^{u(r)})M$. Then there exists a chain complex map $$\xi: P_\bullet(M/(\tildes^{u(r)})M) \to K(\tildes^r;M)$$ such that $H_0(\xi)$ is the natural surjection $M/(\tildes^{u(r)})M \twoheadrightarrow M/(\tildes^r)M$.
  \end{cor}
  \begin{proof}
      By the lifting property for a complex of projective modules \cite[Theorem 2.2.6]{weibelhomological}, there exists a chain complex map $\psi:P_\bullet(M/(\tildes^{u(r)})M) \to T'(\tildes^{u(r)};M)$ such that $H_0(\psi)$ is the identity map on $M/(\tildes^{u(r)})M$. Composing $\psi$ with the map $\phi:T'(\tildes^{u(r)};M) \to K(\tildes^r;M)$ obtained from Theorem \ref{chainmapexists}, we get $\xi = \phi \circ \psi : P_\bullet(M/(\tildes^{u(r)})M) \to K(\tildes^r;M)$. Since $H_0(\phi)$
      is the natural surjection $M/(\tildes^{u(r)})M \twoheadrightarrow M/(\tildes^r)M$, so is $H_0(\xi)$.
  \end{proof}

For the sequence $\tildes = s_1,\dots,s_d$ and an $R$-module $N$, the dual of the Koszul complex $K(\tildes^r;N)^\vee := \Hom(K(\tildes),N)$  is a cochain complex concentrated in degrees zero to $d$. It is well known, for example by applying \cite[Proposition 1.6.9]{Bruns-Herzog} for all $r \geq 1$, that there are natural homomorphisms $H_i(K(\tildes^r;N)) \to \Tor_i^R(R/(\tildes^r),N)$ and $\Ext^i_R(R/(\tildes^r),N) \to H^i(K(\tildes^r;N)^\vee)$ for all $ i \geq 0$. The next corollary says that after raising powers suitably in the $\Tor$ and $\Ext$ modules, the maps in the other direction also exist.
 \begin{cor}\label{connection tor ext}
     Let $I = (\tildes)$ be an ideal and $N$ an $R$-module. Let $r \geq 1$ and $u(r)$ denote the integer as in Theorem \ref{chainmapexists} for $M = R$. Then for all $ i \geq 0$, there exist natural homomorphisms 
     $$ \Tor_i^R(R/(\tildes^{u(r)}),N) \to H_i(K(\tildes^r;N)) \text{\ and\ } H^i(K(\tildes^r;N)^\vee) \to \Ext^i_R(R/(\tildes^{u(r)}),N).$$ 
\end{cor}
\begin{proof}
    Consider the map $\phi$ obtained for $M = R$ in Theorem \ref{chainmapexists}. The functors $- \otimes_R N$ and $\Hom_R(-,N)$ applied to $\phi$ gives the required natural maps.
\end{proof}
The local cohomology module $H^i_I(N)$ can be defined either as $\lim_{\longrightarrow} \Ext^i_R(R/(\tildes^{n}),N)$ , or as \v{C}ech cohomology. The standard proof of its connection with Koszul cohomologies is by noting that the duals of the Koszul complexes $\left\{ K(\tildes^r;N)^\vee \right\}_{r \geq 1}$ converge to the \v{C}ech complex. However, the above discussion yields a new proof of this connection, without using the \v{C}ech complex.
\begin{cor}
    Let $I = (\tildes)$ be an ideal and $N$ an $R$-module. Then
    $$H^i_I(N) = \lim_{\longrightarrow} \Ext^i_R(R/(\tildes^{n}),N) \cong \lim_{\longrightarrow} H^i(K(\tildes^n;N)^\vee).$$
\end{cor}
\begin{rmk}
The picture below summarizes the structure of the proofs thus far.
$$\text{Artin-Rees lemma} \Rightarrow \text{vanishing of Tor maps} \Rightarrow \text{vanishing of Koszul maps} \Rightarrow \text{existence of } \phi $$
Note that once we know the existence of the map $\phi$ in Theorem~\ref{chainmapexists}, it follows immediately that $H_i(\kappa^{u(r),r}) : H_i(K(\tildes^{u(r)}; M)) \to H_i(K(\tildes^r;M))$ is zero for all $i$ since $\kappa^{u(r),r}$ factors through $T'(\tildes^{u(r)};M)$ which has zero homologies for all $i \neq 0$. Thus, the vanishing of the maps of Koszul homologies is equivalent to the existence of the map $\phi$.
\end{rmk}

It is well-known (for example, refer to the proof of Theorem \ref{homology map is zero}) that the Koszul homologies can be described in terms of $\Tor$ modules, which yields the second implication above. The next proposition and theorem show that the vanishing of maps between $\Tor$ modules can be obtained from the vanishing of maps of Koszul homologies, i.e., the second implication can also be reversed.

\begin{prop}\label{koszul-tor vanishing}
    Let $I = (\tildes )$ be an ideal. Let $i \geq 1$ and $r \geq 1$. Then the following are equivalent. 
    \begin{enumerate}[(i)]
        \item There exists $v(i,r) \geq r$ such that the map $\Tor_i^R(R/I^{v(i,r)},M) \to \Tor_i^R(R/I^r,M)$ is zero. 
        \item There exists $w(i,r) \geq r$ such that $H_i(\kappa^{w(i,r),r}) : H_i(K(\tildes^{w(i,r)}; M)) \to H_i(K(\tildes^r;M))$ is zero.
    \end{enumerate}
\end{prop}
\begin{proof}
    $(1) \implies (2):$  By applying Theorem \ref{chainmapexists}, there exists a chain complex map $T (\tildes^{u(r)}) \to K(\tildes^r)$ for some $u(r) \in \BN$. Applying $(1)$ for $u(r)d \in \BN$, we get $v(i,u(r)d)\geq u(r)d$ such that the map 
    \begin{equation}\label{eqn tor vanishing}
        \Tor_i^R(R/I^{v(i,u(r)d)},M) \to \Tor_i^R(R/I^{u(r)d},M)
    \end{equation}
    is zero.  We have natural surjections 
    $$R/(\tildes^{v(i,u(r)d)}) \twoheadrightarrow R/I^{v(i,u(r)d)} \twoheadrightarrow R/I^{u(r)d} \twoheadrightarrow R/(\tildes^{u(r)}).$$ Also note that $K(\tildes^{v(i,u(r)d)}) \subseteq T (\tildes^{v(i,u(r)d)})$. Thus we get natural maps 
    \begin{multline*}
        H_i(K(\tildes^{v(i,u(r)d)});M) \longrightarrow \Tor_i^R(R/(\tildes^{v(i,u(r)d)}),M) \longrightarrow \Tor_i^R(R/I^{v(i,u(r)d)},M) \longrightarrow \\ \Tor_i^R(R/I^{u(r)d},M) \longrightarrow \Tor_i^R(R/(\tildes^{u(r)}),M) \longrightarrow H_i(K(\tildes^{r});M)~,
    \end{multline*} whose composition is $H_i(\kappa^{v(i,u(r)d),r})$.
    Since the map in (\ref{eqn tor vanishing}) is zero, we get the above composition to be zero. Thus defining $w(i,r) = v(i,u(r)d)$, the proof is complete. \\
    
    $(2) \implies (1):$ Given $w(i,r) \in \BN$, by Theorem \ref{chainmapexists}, there exists $u(w(i,r)) \geq w(i,r)$ such that there exists a chain complex map $T (\tildes^{u(w(i,r))}) \xrightarrow{\phi} K(\tildes^{w(i,r)})$. We have natural surjections 
    $$R/I^{u(w(i,r))d} \twoheadrightarrow R/(\tildes^{u(w(i,r))}) \twoheadrightarrow R/(\tildes^r) \twoheadrightarrow R/I^r$$
    which induces the maps 
    \begin{multline*}
        \Tor_i^R(R/I^{u(w(i,r))d},M) \to \Tor_i^R(R/(\tildes^{u(w(i,r))}),M) \to  \Tor_i^R(R/(\tildes^r),M) \to \Tor_i^R(R/I^r,M).
    \end{multline*}
    We will prove that the second map above is zero and hence so is the composition.
    Note that there are  chain complex maps $$T (\tildes^{u(w(i,r))}) \overset{\phi}{\longrightarrow} K(\tildes^{w(i,r)}) \xrightarrow{\kappa^{w(i,r),r}} K(\tildes^r) \lhook\joinrel\longrightarrow T (\tildes^r)$$ such that the  homology maps at the zeroth level are natural surjections. Hence,  the map $\Tor_i^R(R/(\tildes^{u(w(i,r))}),M) \to  \Tor_i^R(R/(\tildes^r),M) $ factors through $H_i(K(\tildes^{w(i,r)}; M)) \to H_i(K(\tildes^r;M))$ which is zero by assumption. Hence by defining $v(i,r) = u(w(i,r))d$, the claim follows. 
\end{proof}
In \cite[Corollary 4.9]{Aberbach}, it is proved that when $R$ is a noetherian local ring, then $M$ is syzygetically Artin-Rees with respect to $I$, i.e., a linear function $w(r)$ exists such that $\Tor_i^R(R/I^{w(r)},M) \to \Tor_i^R(R/I^r,M)$ is zero for all $i \geq 1$. The next theorem provides a polynomial function $w(r)$, even for rings which are not necessarily local, such that the maps between $\Tor$ modules of degrees $w(r)$ to $r$ vanish. Note that the function $w(r)$ is independent of the degree $i$, which enhances the degree-wise vanishing of maps between $\Tor$ modules in \cite[Chapter 10, Proposition 10 and Lemma 11 ]{andre_michel_commalg}.
\begin{thm}\label{uniform artin rees}
    Let $R$ be a noetherian ring, $I = (\tildes)$ an ideal and $M$ a finitely generated $R$-module. Given $r \geq 1$, there exists $w(r):=u(h+rd)$ such that  the map $$\Tor_i^R(R/I^{w(r)},M) \to \Tor_i^R(R/I^r,M)$$ is zero for all $i \geq 1$.
\end{thm}
\begin{proof}
    By Lemma \ref{homology map is zero}, the map $H_i(\kappa^{h+rd,r}): H_i(K(\tildes^{h+rd};M)) \to H_i(K(\tildes^r;M))$ is zero for all $i>0$. By Theorem \ref{chainmapexists}, there exists a chain complex map $T (\tildes^{u(h+rd)}) \to K(\tildes^{h+rd})$. Define $w(r) := u(h+rd)$. Thus we get a map of chain complexes $$T (\tildes^{w(r)}) \otimes M \to K(\tildes^{h+rd};M) \to K(\tildes^r;M) \to T (\tildes^r) \otimes M.$$ Since the map on homologies of the middle map is zero, we get $$\Tor_i^R(R/I^{w(r)},M) \to \Tor_i^R(R/I^r,M)$$ is the zero map for all $i>0$. 
\end{proof}
\begin{cor}
 Let $I$ be a principal ideal and $M$ a finitely generated $R$-module. Then the map $\Tor_i^R(R/I^{2h+l+r},M) \to \Tor_i^R(R/I^r,M)$ is zero for all $i>0$. Thus $M$ is syzygetically Artin-Rees with respect to $I$.
\end{cor}
\begin{proof}
    When $I$ is a principal ideal, $d= 1$. So we get $w(r) = u(h+r) = 2h + l + r$ and hence the claim.
\end{proof}

\begin{lemma}\label{betabound}
    Let $\ch(R) = p >0$, where $p$ is a prime and $I = (\tildes)$ an ideal in $R$. Then 
    \begin{enumerate}
        \item Every $P_\bullet \in Ch^{[0,k]}_{\tiny \L}(\CP)$ where $k > \pd(R/I)$ has a reducer.
        \item $\beta(\L_{V(I)}) \leq \pd(R/I)$.
    \end{enumerate}
\end{lemma}
\begin{proof}
     From Lemma \ref{Foxby-Halvorsen}, there exists a chain complex map $\Psi : K(\tildes^r;P_0)  \to P_\bullet$ for some $r \in \BN$ such that $\Psi_0$ is an isomorphism. Consider the map $\phi: T (\tildes^u) \to K(\tildes^r)$ for some $u \in \BN$ constructed in Theorem \ref{chainmapexists}. Without loss of generality, we can assume that $u = p^n$ for some $n \in \BN$. From Lemma \ref{char p proj dim}, $\pd(R/I) = \pd(R/I^{[p^n]})$ for all $n \in \BN$.
     Hence the canonical truncation $T'_{\bullet}$ of $T(\tildes^u)$ 
     at that degree is homotopy equivalent to it via the natural inclusion. Define $T_{\bullet} = T'_{\bullet} \otimes P_0$ 
    and $\alpha$ to be the chain complex map obtained by composition as follows: $$T_{\bullet} \hookrightarrow T(\tildes^u) \otimes P_0 \xrightarrow{\phi \otimes id_{P_0}}  K(\tildes^r;P_0) \xrightarrow{\Psi} P_{\bullet} \quad .$$ It is now straightforward to see that $T_0 = K_0(\tildes^r;P_0)$ and hence $\alpha_0$ is surjective.
    Hence whenever $k > \pd(R/I) = \pd(R/I^{[p^n]})$, $(T_{\bullet},\alpha)$ is a reducer of $P_\bullet$. It follows from the definition that $\beta(\L_{V(I)}) \leq \pd(R/I)$. 
    \end{proof}

    \begin{thm}\label{k0iso}
     Let $\ch(R) = p >0$, where $p$ is a prime and $\L = \L_{V(I)}$ where $I$ is a perfect ideal in $R$. Then $K_0(\xbar{\CP} \cap \L) \cong K_0(\DbLP).$
    \end{thm}
    \begin{proof}
     By Lemma \ref{beta properties} and Lemma \ref{betabound}, $\grade(I) \leq \beta(\L) \leq \pd(R/I)$. Since $I$ is a perfect ideal, $\pd(R/I) = \grade(I)$. Therefore $\beta(\L) = \grade(I)$. Applying Lemma \ref{grade-beta equality}, we get $K_0(\xbar{\CP} \cap \L) \cong K_0(\DbLP).$
    \end{proof}

\section{A Derived Equivalence}\label{section der-equiv}
The main results in this section largely follow along the lines of \cite[Sections 3 and 4 ]{SandersSane}.
As mentioned in the introduction, the proofs go through with the role of the Koszul resolution in that article being played by the Tate resolution here. Unfortunately, the statements (specifically the crucial \cite[Theorems 3.1 and 3.3]{SandersSane}) there do not allow for direct application to our context since they use the structure of the Koszul resolution in the proofs.

We first define a strong reducer, which essentially collects the properties of the Koszul resolution used in the proofs in \cite[Sections 3 and 4 ]{SandersSane}.
\begin{defn}\label{SR prop}
    Let $\CA \subseteq \modr$ be a resolving subcategory, $\L$ a Serre subcategory of $\modr$, $X_\bullet \in Ch^b_{\tiny \L} (\modr)$ and $J \subseteq R$ such that $X_\bullet$ is not exact and $R/J \in \L$. We say that $(T_\bullet,\alpha,I)$  is a strong reducer of $(X_\bullet, J)$ with respect to $\CA$
    if the following conditions are satisfied : 
    \begin{enumerate}
        \item $T_{\bullet} \in Ch^b_{\tiny \L}(\CA)$.
        \item $\min_c(T_{\bullet}) = m$ where $m = \min(X_\bullet)$.
        \item $\alpha: T_{\bullet} \to X_{\bullet}$ is a chain complex map.
        \item $\supph(T_{\bullet}) = \{ m \}$ and $H_m(T_\bullet) = T_m/IT_m$.
        \item $H_m(\alpha): H_m(T_{\bullet}) \to H_m(X_{\bullet})$ is surjective.
        \item $I \subseteq J$ .
    \end{enumerate}
    \indent A Serre subcategory $\L$ of $\modr$ is defined to have the strong reducer (SR) property with respect to $\CA$, if for every $X_\bullet \in Ch^b_{\tiny \L} (\modr)$ and $J \subseteq R$ such that $X_\bullet$ is not an exact complex and $R/J \in \L$, $(X_\bullet, J)$ has a strong reducer with respect to $\CA$.
\end{defn}

For the rest of the section, we use the following notations: $\L$ is a Serre subcategory of $\modr$, $\CA$ a resolving subcategory of $\modr$ and $\T$ a thick subcategory of $\modr$ containing $\CA$.
The following lemma is analogous to \cite[Lemma 3.2]{SandersSane}.
\begin{lemma}\label{width lemma 1}
    Let $X_\bullet \in Ch^b_{\tiny \L} (\modr)$ and $J \subseteq R$ such that $R/J \in \L$. Suppose $(X_\bullet, J)$ has a strong reducer $(T_\bullet,\alpha,I)$ with respect to $\CA$. 
    Extend the chain complex map $\alpha$ to an exact triangle in $D^b_{\tiny \L} (\modr)$ : 
$$\Sigma^{-1} C_\bullet \xrightarrow[]{} T_\bullet \xrightarrow[]{\alpha} X_\bullet \xrightarrow{}C_\bullet.$$
If $\width(X_\bullet) >0$, then we have the following:
\begin{enumerate}
    \item $\width(C_\bullet) < \width(X_\bullet)$.
    \item $\width(\Sigma^{-1} C_\bullet \oplus T_\bullet) < \width(X_\bullet)$.
\end{enumerate}
\end{lemma}
\begin{proof}
    Let $m = \min(X_\bullet)$. Since $H_i(T_\bullet) = 0$ for all $i \neq m$, we have $H_i(X_\bullet) \cong H_i(C_\bullet)$ for all $i \neq m, m+1$. Since $H_m(\alpha)$ is surjective,  $H_m(C_\bullet) = 0$. Hence the claim follows.
\end{proof}

The following lemma is similar to \cite[Theorem 3.3 and Lemma 3.4]{SandersSane}. The proof is really a check that the proofs of those results work, with the strong reducer replacing the 
choices of ideals and Koszul resolution. That being the case, we only give a sketch of the proof.
\begin{lemma}\label{morphisms in derived category} 
    Suppose $\L$ has the SR property with respect to $\CA$. Let $X_\bullet \xrightarrow{f} Y_\bullet$ be a morphism in $\DbLT$ such that $X_\bullet$ and $Y_\bullet$ are chain complexes in $Ch^b(\T \cap \L), \min(X_\bullet \oplus Y_\bullet) =m$ and $\min_c(X_\bullet), \min_c(Y_\bullet) \geq m$. Then there exists chain complexes $M_\bullet^X$ and $M_\bullet^Y$ in $Ch^b(\T \cap \L)$ and maps of chain complexes $M_\bullet^X \xrightarrow{\beta^X} X_\bullet$, $M_\bullet^Y \xrightarrow{\beta^Y} Y_\bullet$ and $M_\bullet^X \xrightarrow{\rho} M_\bullet^Y$, such that 
    \begin{itemize}
        \item $M_\bullet^X, M_\bullet^Y \in Ch^b(\T \cap \L)$ with $M_i^X = M_i^Y = 0$ for all $i \neq m$.
        \item there is a commutative square in $\DbLT$: 
        $$\xymatrix{ M_\bullet^X \ar[r]^{\beta^X} \ar[d]^{\rho} & X_\bullet \ar[d]^{f\qquad .}\\ M_\bullet^Y \ar[r]^{\beta^Y} & Y_\bullet }$$
        \item $H_m(\beta^X)$ and $H_m(\beta^Y)$ are surjective. 
    \end{itemize}
    Let $C_\bullet^X$ and $C_\bullet^Y$ be the cones of $\beta^X$ and $\beta^Y$ respectively. If $\width(X_\bullet \oplus Y_\bullet) = k >0$, then
    \begin{enumerate}
        \item $\width(C_\bullet^X \oplus C_\bullet^Y)< k$.
        \item $\width(C_\bullet^X \oplus Y_\bullet)\leq  k$.
        \item If $\min(X_\bullet) < \min(Y_\bullet)$, then $\min(C_\bullet^X \oplus Y_\bullet)\leq  k-1$.
    \end{enumerate}
\end{lemma}
\begin{proof}
Define the complex $M_\bullet^Y$ to be $\Sigma^m Y_m$ and $\beta^Y: M_\bullet^Y \to Y_\bullet$ the inclusion map. Then $M_\bullet^Y \in Ch^b(\T \cap \L)$ and $\beta^Y$ is a chain complex map such that $H_m(\beta^Y)$ is surjective. \\
Let $f$ be given by a roof diagram $X_\bullet \xleftarrow{q} Q_\bullet \xrightarrow{g} Y_\bullet$ where $q$ is a quasi-isomorphism. We may assume $\min_c(Q_\bullet) = m$. Let $Q_\bullet'$ be the pullback of $\beta^Y$ and $g$ in $Ch^b(\modr)$: 
$$\xymatrix{ Q_\bullet' \ar@{..>}[r]^{\nu} \ar@{..>}[d]^{\mu} & Q_\bullet\ar[d]^{g  \qquad .} \\ M_\bullet^Y \ar[r]^{\beta^Y} & Y_\bullet }$$
Thus we get an exact sequence $0 \to Q_\bullet' \to Q_\bullet \oplus M_\bullet^Y \to B_\bullet \to 0$ where $B_\bullet$ is the image of the map $Q_\bullet \oplus M_\bullet^Y \xrightarrow{(g, -\beta^Y)} Y_\bullet$. We observe $B_\bullet \in Ch^b(\L)$ and hence $B_\bullet \in Ch^{b}_{\tiny \L}(\modr)$. Since $M_\bullet^Y$ and $Q_\bullet$ also lie in $Ch^{b}_{\tiny \L}(\modr)$, we get $Q_\bullet' \in Ch^{b}_{\tiny \L}(\modr)$. Set $\lambda = q \circ \nu$ and $J = \Ann(X_m)$. Then $R/J \in \L$. Since $\L$ has the SR property with respect to $\CA$, $(Q_\bullet',J)$ has a strong reducer $(T_\bullet,\alpha,I)$ with respect to $\CA$. Therefore, we have the following commutative diagram in $D^{b}_{\tiny \L}(\modr)$: 
$$\xymatrix{ T_\bullet \ar[r]^{\alpha} & Q_\bullet' \ar[r]^{\lambda} \ar[d]^{\mu} & X_\bullet\ar[d]^{f \qquad .} \\ & M_\bullet^Y \ar[r]^{\beta^Y} & Y_\bullet }$$
Define $M_\bullet^X$ to be the chain complex $\Sigma^m H_m(T_\bullet) \in Ch^b(\xbar \CA \cap \L) \subseteq Ch^b(\T \cap \L)$. Since $I \subseteq J = \Ann(X_m)$ and $H_m(T_\bullet) \cong T_m/IT_m$, the map $\lambda_m \circ \alpha_m : T_m \to X_m$ factors through $T_m/IT_m$. Therefore, we get a chain complex map $\beta^X : M_\bullet^X \to X_\bullet$. Also $\mu \circ \alpha$ factors through $M_\bullet^X$ giving a chain complex map $\rho: M_\bullet^X \to M_\bullet^Y$ such that the following diagram commutes in $\DbLT$:
 $$\xymatrix{ M_\bullet^X \ar[r]^{\beta^X} \ar[d]^{\rho} & X_\bullet \ar[d]^{f  \qquad .} \\ M_\bullet^Y \ar[r]^{\beta^Y} & Y_\bullet. }$$To show that $H_m(\beta^X)$ is surjective, it is enough to show that $H_m(\nu)$ is surjective, which follows because $Q'_{\bullet}$ is just the degree-wise pullback. The statements $(1), (2)$ and $(3)$ can be derived from the long exact sequence of homologies arising from the mapping cone, similar to the proof of the previous lemma.
\end{proof}

We now state the technical heart of this section which is analogous to the key theorem \cite[Theorem 4.5]{SandersSane} of \cite{SandersSane}.
Let $\iota: D^b(\T \cap \L) \rightsquigarrow \DbLT$ be the natural functor induced by the inclusion map $Ch^b(\T \cap \L) \hookrightarrow \ChbLT$. 
\begin{thm}\label{general derived equivalence}
    Suppose $\L$ has the SR property with respect to $\CA$. Then the functor $\iota: D^b(\T \cap \L) \rightsquigarrow \DbLT$ is an equivalence of categories. 
\end{thm}
\begin{proof}
    Since $\iota$ is a triangulated functor between triangulated categories, it suffices to prove that $\iota$ is faithful on objects, full and essentially surjective (for example by \cite[Lemma 2.19]{SandersSane}). If $\iota(X_\bullet) = 0$, i.e., it is acyclic, then $X_\bullet$ is also acyclic. Hence $\iota$ is faithful on objects and it suffices to prove that $\iota$ is essentially surjective and full.
    This proof is verbatim the same as that of \cite[Proposition 4.4]{SandersSane} and hence, we only briefly mention the main steps in the proof and request the interested reader to read the details from \cite{SandersSane}. The following induction statement on $k$ will show that $\iota$ is essentially surjective and full :
    \begin{enumerate}
        \item For every $P_\bullet \in \DbLT$ with $\width(P_\bullet) = k$, there exists $\tilde{P_\bullet} \in D^b(\T \cap \L)$ such that $\iota(\tilde{P_\bullet}) \cong P_\bullet$. 
        \item $\Hom_{\tiny D^b(\T \cap \L)}(X_\bullet,Y_\bullet) \to \Hom_{\tiny \DbLT}(X_\bullet,Y_\bullet)$ is surjective for all $X_\bullet,Y_\bullet \in D^b(\T \cap \L)$ such that $\width(X_\bullet \oplus Y_\bullet) = k$.
    \end{enumerate}
    The base case is when $k = 0$. If $\width(P_\bullet) = 0$, then $\supph(P_\bullet) \subseteq \{ m \}$ for some $m \in \BZ$. Thus $P_\bullet$ is quasi-isomorphic to $\Sigma^m H_m(P_\bullet)$. It follows from the definitions of $\T$ and $\L$ that $H_m(P_\bullet) \in \T \cap \L$ and hence $\Sigma^m H_m(P_\bullet) \in D^b(\T \cap \L)$, which proves (1). Similarly $\width(X_\bullet \oplus Y_\bullet) = 0$ implies that $X_\bullet$ and $Y_\bullet$ are quasi-isomorphic to $\Sigma^m H_m(X_\bullet)$ and $\Sigma^m H_m(Y_\bullet)$ for some $m \in \BZ$ respectively. Note that for any $R$-modules $M$ and $N$, regarding them as complexes, we have
    $$\Hom_{\tiny \T \cap \L}(M,N) \xrightarrow{\sim} \Hom_{\tiny D^b(\T \cap \L)}(M,N) \xrightarrow[]{\sim} \Hom_{\tiny \DbLT}(\iota(M),\iota(N)) \xrightarrow[]{\sim} \Hom_R(M,N).$$ Hence (2) follows when $k = 0$. \\
    \indent Assume the statement holds for all $k' <k$. Since $\L$ has the SR property with respect to $\CA$, $(P_\bullet,R)$ has a strong reducer $(Q_\bullet,\beta,I)$ with respect to $\CA$. Therefore, $Q_\bullet \in Ch^b_{\tiny \L}(\CA) \subseteq Ch^b_{\tiny \L}(\T)$ and $\supph(Q_\bullet) = \{ m\}$. Extending the chain complex morphism $\beta: Q_\bullet \to P_\bullet$, we get an  exact triangle $\Sigma^{-1} C_\bullet \xrightarrow[]{\alpha} Q_\bullet \xrightarrow[]{\beta} P_\bullet \xrightarrow{\gamma}C_\bullet$. Lemma \ref{width lemma 1} shows that $\width(C_\bullet) < k$ and $\width(\Sigma^{-1} C_\bullet \oplus Q_\bullet) < k$. Applying the induction hypothesis (1) and (2), there exists $\tilde{C_\bullet} \in D^b(\T \cap \L)$ such that $\iota(\tilde{C_\bullet}) = C_\bullet$ and the map $$\Hom_{\tiny D^b(\T \cap \L)}(\Sigma^{-1} \tilde{C_\bullet},\Sigma^{m}H_m(Q_\bullet)) \to \Hom_{\tiny \DbLT}(\Sigma^{-1} \tilde{C_\bullet},\Sigma^{m}H_m(Q_\bullet)) \cong \Hom_{\tiny \DbLT}(\Sigma^{-1} \tilde{C_\bullet}, Q_\bullet)$$ is surjective. Thus there exists $\tilde{\alpha}: \Sigma^{-1} \tilde{C_\bullet}\to \Sigma^{m}H_m(Q_\bullet)$  with cone $\tilde{P_\bullet}$ such that the following diagram commutes:
$$\xymatrix{ \Sigma^{-1} C_\bullet \ar[r]^{\alpha} \ar[d]^*[@]{\hspace{-1.5mm}\sim} & Q_\bullet \ar[d]^*[@]{\hspace{-1.5mm}\sim} \ar[r]^{\beta} & P_\bullet \ar[r]^{\gamma} & C_\bullet \ar[d]^*[@]{\hspace{-1.5mm}\sim}\\ 
\iota(\Sigma^{-1} \tilde{C_\bullet}) \ar[r]^-{\iota(\tilde{\alpha})} & \iota(\Sigma^{m}H_m(Q_\bullet)) \ar[r]^-{\iota({\tilde{\beta}})} & \iota({\tilde{P_\bullet}}) \ar[r]^{\iota({\tilde{\gamma}})} & \iota({\tilde{C_\bullet}})}$$
By the axiom TR3 of triangulated categories, $\iota({\tilde{P_\bullet}}) \cong P_\bullet$. Hence (1) is proved. \\
\indent We now give an idea of the proof of (2). Let $X_\bullet, Y_\bullet \in D^b(\T \cap \L)$ and $f \in \Hom_{\tiny \DbLT}(X_\bullet,Y_\bullet)$ and $\min(X_\bullet \oplus Y_\bullet) = m$. We may assume that $\min_c(X_\bullet ) \geq m$ and $\min_c(Y_\bullet ) \geq m$. Applying Lemma \ref{morphisms in derived category}, we get a morphism of triangles in $\DbLT$: 
$$\xymatrix{ \Sigma^{-1}C_\bullet^X \ar[r]^-{\alpha^X} \ar[d]^{\Sigma^{-1} \lambda} & M_\bullet^X \ar[r]^{\beta^X} \ar[d]^{\rho} & X_\bullet \ar[r]^{\gamma^X} \ar[d]^{f} & C_\bullet^X \ar[d]^{\lambda}\\ \Sigma^{-1}C_\bullet^Y \ar[r]^-{\alpha^Y} & M_\bullet^Y \ar[r]^{\beta^Y} & Y_\bullet \ar[r]^{\gamma^Y} &C_\bullet^Y }$$
where $M^X_{\bullet} , M^Y_{\bullet}$ are concentrated in degree $m$, the morphisms $\beta^X, \beta^Y, \rho $ are chain complex maps, $H_m(\beta^X), H_m(\beta^Y)$ are surjective and $\width(C_\bullet^X \oplus C_\bullet^Y) < k$. Applying the induction hypothesis (2) to $C_\bullet^X$ and  $C_\bullet^Y$, we get $\tilde{\lambda} \in \Hom_{\tiny D^b(\T \cap \L)}(C_\bullet^X, C_\bullet^Y)$ such that $\iota(\tilde{\lambda}) = \lambda$. We thus get a commutative diagram in $D^b(\T \cap \L)$:
$$\xymatrix{ \Sigma^{-1}C_\bullet^X \ar[r]^-{\alpha^X} \ar[d]^{\Sigma^{-1} \tilde{\lambda}} & M_\bullet^X \ar[r]^{\beta^X} \ar[d]^{\rho} & X_\bullet \ar@{-->}[d]^{g} \ar[r]^{\gamma^X}  & C_\bullet^X \ar[d]^{\tilde{\lambda}}\\ \Sigma^{-1}C_\bullet^Y \ar[r]^-{\alpha^Y} & M_\bullet^Y \ar[r]^{\beta^Y} & Y_\bullet \ar[r]^{\gamma^Y} &C_\bullet^Y }$$
where $g$ exists by axiom TR3 of triangulatedcategories. Unfortunately, it does not follow that $\iota(g) = f$. By linearity, we get a morphism of triangles in $\DbLT$ as above with the vertical arrow $f - \iota(g)$ from $X_\bullet$ to $Y_\bullet$ and other vertical arrows being zero. Using the weak kernel/cokernel properties, we get maps from $C_\bullet^X$ to $Y_\bullet$ and $X_\bullet$ to $M^Y_\bullet$. A very careful analysis of the Hom sets between these complexes and the induction hypothesis finally yields that $f = \iota(h)$ for some $h$, thus completing the proof.
\end{proof}

The following theorem is an immediate consequence of Theorem \ref{general derived equivalence}.
\begin{thm}\label{derivedequivalence}
     Suppose $\L$ has the SR property with respect to $\CA$. Then there is an equivalence of categories $ D^b(\xbar{\CA} \cap \L) \rightsquigarrow {D^b_{\tiny \L}}(\CA)$. In particular, $ D^b( \xbar{\CP} \cap \L ) \simeq \DbLP $.
\end{thm}
\begin{proof}
    Applying Theorem \ref{general derived equivalence} with $\T = \xbar{\CA}$, we get that
    $D^b(\xbar{\CA} \cap \L) \simeq {D^b_{\tiny \L}}(\xbar{\CA})$.
    For a resolving subcategory $\CA$, the natural functor ${D^b_{\tiny \L}}(\CA) \rightsquigarrow {D^b_{\tiny \L}}(\xbar{\CA})$ is an equivalence using a standard total complex argument. Hence $D^b(\xbar{\CA} \cap \L) \simeq {D^b_{\tiny \L}}(\CA)$.
\end{proof}
The following lemma gives a sufficient condition for a Serre subcategory to have the SR property in terms of ideals having filtration of finite $\CA$-dimension ideals.
\begin{lemma}\label{koszul trick}
    Let $\L$ be a Serre subcategory with the property that whenever $R/I' \in \L$, there exists $I \subseteq I'$ such that $R/I \in \L$ and $I$ has a filtration $\{ I_n \}$ which is equivalent to the $I$-adic filtration and $\textrm{dim}_{\tiny \CA}(R/I_n) < \infty$. Then $\L$ has the SR property with respect to $\CA$.
\end{lemma}
\begin{proof}
    Let $X_{\bullet} \in Ch^b_{\tiny \L} (\modr)$ and $J \subseteq R$ an ideal such that $X_\bullet$ is not exact and $R/J \in \L$. Without loss of generality, assume $\min(X_\bullet) = 0$. Then there exists a complex $F_{\bullet}$ of finitely generated free modules and a quasi-isomorphism $\pi: F_{\bullet} \to X_{\bullet}$ where $F_i = 0$ for all $i<0$ and $\pi_0$ is surjective. It follows from \cite[Corollary 10.4.7]{weibelhomological} that $$\End_{K(R)}(F_\bullet) \cong \End_{D(R)}(F_\bullet) \cong \Hom_{D(R)}(F_\bullet,X_\bullet) \cong \Hom_{K(R)}(F_\bullet,X_\bullet).$$ Since $X_\bullet$ is bounded, $\Hom_{K(R)}(F_\bullet,X_\bullet)$ is finitely generated and hence, so is $\End_{K(R)}(F_\bullet)$. Also, the map  $(\End_{K(R)}(F_\bullet))_{\cp} \to \End_{K(R_{\tiny \cp})}((F_\bullet)_p)$ is injective for all $\cp \in Spec(R)$, which follows from the above isomorphism and boundedness assumption on $X_\bullet$. 
    \par We claim that $\Supp(\End_{K(R)}(F_\bullet)) \subseteq \bigcup_{i = 0}^{\infty} \Supp(H_i(F_\bullet))$. Let $\cp \in Spec(R)$ such that $(\End_{K(R)}(F_\bullet))_{\cp} \neq 0$. Then $\End_{K(R_{\tiny \cp})}((F_\bullet)_{\cp}) \neq 0$ and hence $(F_\bullet)_{\cp}$ is not homotopic to $O_\bullet$. Therefore $(F_\bullet)_{\cp}$ is not acyclic, since $F_\bullet$ is bounded below. So $\cp \in \Supp(H_i(F_\bullet))$ for some $i$ and hence the claim. Therefore we get 
    $$\Supp(\End_{K(R)}(F_\bullet)) \subseteq \bigcup_{i = 0}^{\infty} \Supp(H_i(F_\bullet)) = \bigcup_{i = 0}^{\infty} \Supp(H_i(X_\bullet)) \subseteq V_{\tiny \L}$$ where $V_{\tiny \L}$ is the specialization closed set corresponding to $\L$ as defined in \ref{serresub defn}(\ref{VL defn}).
    \par Define $N = \Ann(\End_{K(R)}(F_{\bullet})) \subseteq R$. Then $V(N) = \Supp(\End_{K(R)}(F_\bullet)) \subseteq V_{\tiny \L}$. Since we assumed $R/J \in \L$, $V(J) = \Supp(R/J) \subseteq V_{\tiny \L}$. Hence $$V(N \cap J) = V(N) \cup V(J) \subseteq V_{\tiny \L}.$$ Therefore, $R/(N \cap J) \in \L$. By the assumption on $\L$, there exists $I \subseteq (N \cap J)$ such that  $R/I \in \L$ and $I$ has a filtration $\{ I_n \}$ which is equivalent to the $I$-adic filtration and $\textrm{dim}_{\tiny \CA}(R/I_n) < \infty$. Let $f_1,\dots,f_t$ generate $I$. Since $I \subseteq N' = \Ann(\End_{K(R)}(F_{\bullet}))$, $f_i{id}_{\tiny F_{\bullet}}$ is null-homotopic for each $i$. By Lemma \ref{Foxby-Halvorsen}, there exists a
    chain complex map $$\Psi: K_{\bullet} = K(f_1,\dots,f_t) \otimes_R F_0 \to F_{\bullet}$$ such that $\min_c(K_{\bullet}) = 0$ and $\Psi_0: K_0 \to F_0$ is an isomorphism.
    Hence $H_0(\Psi)$ is surjective. 
    \par From Theorem \ref{chainmapexists}, there exist $u \in\BN$ and a chain complex map $\phi: T (f_1^u,\dots,f_t^u)  \to K(f_1,\dots,f_t)$ such that $\phi_0$ is an isomorphism and hence $H_0(\phi)$ is surjective.
    Since the filtration $\{ I_n\}_{n \in \tiny \BN}$ of $I$ equivalent to the filtration $\{I^l\}_{l \in \tiny \BN}$ and hence equivalent to the square-power filtration $\{I^{[l]}\}_{l \in \tiny \BN}$, there exists $s \in \BN$ such that $I_s \subseteq I^{[u]}$. This induces a chain complex map from the Tate resolution of $R/I_s$ to the Tate resolution of $R/I^{[u]}$. Since $\dim_{\tiny \CA}(R/I_s) < \infty$, the  canonical truncation $T'_{\bullet}$ of its Tate resolution at that degree is homotopy equivalent to it via the natural inclusion and consists of modules in $\CA$. Define $T_{\bullet} = T'_{\bullet} \otimes F_0$ and $\alpha$ to be the chain complex map obtained by composition as follows: $$T_{\bullet} \hookrightarrow T (R/I_s) \otimes F_0 \rightarrow  T (f_1^u,\dots,f_t^u) \otimes_R F_0 \xrightarrow{\phi \otimes id_{F_0}} K(f_1,\dots,f_t) \otimes F_0 \xrightarrow{\Psi}
    F_{\bullet} \xrightarrow{\pi} X_{\bullet} \quad .$$ It is now straightforward to see that $T_0 = F_0$ and hence $\alpha_0$ is a surjection. Hence $H_0(\alpha)$ is surjective. Note that $\min_c(T_{\bullet}) = 0$ and  $T_{\bullet}$ is a finite $\CA$-resolution of  $R/I_{s} \otimes_R F_0$. Hence, $T_{\bullet} \in Ch^b_{\tiny \L}(\CA)$, and $\supph(T_{\bullet}) = \{ 0 \}$. Further, it follows that $H_0(T_\bullet) \cong  T_0 /I_sT_0 $. Finally, $I_s \subseteq I^{[u]} \subseteq I \subseteq J$.
    Therefore, $(T_\bullet, \alpha,I_s)$ is a strong reducer of $(X_\bullet,J)$ with respect to $\CA$.
\end{proof}
Finally, as a straightforward consequence of Theorem~\ref{derivedequivalence} and Lemma \ref{koszul trick}, we obtain the main theorem of this section, which appears as Theorem~\ref{intro-der-equiv}.
\begin{thm}\label{der-equiv-main}
Let $R$ be a commutative noetherian ring. Let $\CA \subseteq \modr$ be a resolving subcategory. Then there is an equivalence of categories $ D^b(\xbar{\CA} \cap \L) \rightsquigarrow {D^b_{\tiny \L}}(\CA)$ in the following cases :
    \begin{enumerate}
        \item $R$ is a regular ring.
        \item $\L$ satisfies condition $(*)$ (as defined in the introduction).
        \item $\L = \L_{V(I)}$ where $I \subseteq R$ is an ideal which has a filtration $\{ I_n \}$ which is equivalent to the $I$-adic filtration and $\textrm{dim}_{\tiny \CA}(R/I_n) < \infty$.
    \end{enumerate}
\end{thm}
\begin{proof}
    It is easy to check that in the mentioned cases, $\L$ satisfies the hypothesis of the Lemma \ref{koszul trick} and thus has the SR property with respect to $\CA$. Applying Theorem~\ref{derivedequivalence} completes the proof.
\end{proof}
For the sake of completeness, we restate Corollary \ref{intro-cor} which is an immediate and particularly interesting consequence of Theorem \ref{der-equiv-main}.
\begin{cor}\label{derived equivalence for efpd}
\begin{enumerate}[(a)]
    \item Let $\L = \L_{V(I)}$ where $I$ has efpd. Then $ D^b( \xbar{\CA} \cap \L ) \simeq {D^b_{\tiny \L}}(\CA) $.
    \item If $R$ is of prime characteristic and $\pd(R/I) < \infty$, then $I$ has efpd and hence the above derived equivalence holds.
    \item When $R$ is a regular ring or $\L$ satisfies condition {\rm(*)} or $\L = \L_{V(I)}$ where $I$ has efpd, $ D^b( \xbar{\CP} \cap \L ) \simeq \DbLP $.
\end{enumerate}
\end{cor}

\section{Remarks on ideals of eventually finite projective dimension}\label{last section}
\subsection*{Some consequences of the derived equivalence} 
 We now study when an ideal $I$ in a noetherian ring has efpd. The next lemma exhibits a sequence of implications that give necessary conditions for an ideal to have eventually finite projective dimension.
\begin{lemma}\label{4 statements}
Let $R$ be a noetherian ring, $I \neq R$  an ideal in it and $\L =  \L_{V(I)}$. Consider the following statements.
\begin{enumerate}[(i)]
    \item $I$ is an ideal of eventually finite projective dimension.
    \item There is a derived equivalence $ D^b( \xbar{\CP} \cap \L )  \simeq \DbLP$.
    \item There exists a non-zero finitely generated module $M$ with finite projective dimension and $\Supp(M) = V(I)$.
    \item For every minimal prime $\cp$ of $I$, $R_{\cp}$ is Cohen-Macaulay. 
\end{enumerate}
Then $(i) \Rightarrow (ii) \Rightarrow (iii) \Rightarrow (iv)$.
\end{lemma}
\begin{proof}
    $(i) \Rightarrow (ii)$ : This is Corollary \ref{derived equivalence for efpd}(c). \\
    $(ii) \implies (iii)$ : By the Hopkins-Neeman theorem, $$V(I) = \bigcup_{\tiny{X \in \DbLP}} \Supp(X) = \bigcup_{\tiny{X \in D^b( \xbar{\CP} \cap \L )}} \Supp(X) = \bigcup_{\tiny{M \in \xbar{\CP} \cap \L}} \Supp(M).$$
Hence, for each $\cp_i \in Min(I)$, there exists $M_i \in \xbar{\CP} \cap \L$ such that ${(M_i)}_{\cp_{i}} \neq 0$ . Thus $M = \oplus M_i$ is finitely generated, has finite projective dimension and $\Supp(M) = V(I)$. \\
$(iii) \implies (iv):$ For any minimal prime $\cp$ of $I$, the $R_{\cp}$-module $M_{\cp}$ has finite projective dimension and $\Supp(M_{\cp}) = \{ \cp R_{\cp} \}$. Hence by Corollary \ref{intersection cor}, $R_{\cp}$ is Cohen-Macaulay. 
\end{proof}
Note that the implication $(i) \implies (iii)$ follows directly from the definition (for example, by choosing $M = R/I_n$).
\begin{eg}
    Example \ref{efpd non-example} in which $R=k[X,Y]/(XY)$ for a field $k$ and $I = (X)$ also shows that $(iv) \not\Rightarrow (iii)$.
\end{eg}
\begin{question}
    Are any of the other reverse implications true?
\end{question}
 We highlight below the special case where $(R,\m)$ is a local ring and $I$ is an $\m$-primary ideal wherein the statements above are equivalent.
This is a direct extension of \cite[Theorem 1.1]{SandersSane}, which was the inspiration for the previous lemma, and adds to the characterization of Cohen-Macaulay local rings in terms of the efpd property.
\begin{cor}\label{CM characterization}
    Let $(R,\m)$ be a noetherian local ring. Then the following are equivalent :
    \begin{enumerate}[(i)]
        \item $\m$ has efpd.
        \item $ D^b( \xbar{\CP} \cap \L )  \simeq \DbLP$ where $\L = \L_{V(\m)}$.
        \item There exists a non-zero finitely generated module $M$ with finite length finite projective dimension.
        \item $(R,\m)$ is Cohen-Macaulay.
    \end{enumerate}
\end{cor}
\begin{proof} By the previous lemma, we only need to prove $(iv) \Rightarrow (i)$ which follows by observing that when $R$ is Cohen-Macaulay, then $\m$ is a set-theoretic complete intersection ideal. \\
\end{proof}
Note that some of these implications can be directly proved to be equivalent.

\subsection*{Known results and counter-examples related to ideals having efpd}
The definition of ideals of eventually finite projective dimension raises some natural questions regarding the eventual behaviour of the projective dimension of a filtration. The following 
consequence of \cite[Theorem 2(i)]{Brodmann} is further suggestive that the limiting behaviour may be interesting.
\begin{thm}\label{Brodmann}
    Let $R$ be a noetherian ring and $I$ an ideal in $R$. Then $\pd(R/I^n)$ is constant (possibly infinite) for all large $n \in \BN$. 
\end{thm}
Unfortunately, this property is special to the power filtration and is not shared by other equivalent filtrations such as the symbolic power filtrations, as demonstrated by the following recent result.
\begin{thm*}\label{trung}\cite[Theorem 6.3]{Nguyen-Trung}
    For any positive numerical function $\psi(t)$ which is periodic for $t \gg 0$, there exists a polynomial ring $R$ and a homogeneous ideal $I \subset R$ such that $\pd(R/I^{(n)}) = \psi(n) + c$ for some $c \geq 0$.
\end{thm*}
Let $R$ be the polynomial ring, and $I$ the ideal obtained from the conclusion of the above theorem when $\psi$ be a periodic function having period exactly $2$. Clearly $I$ has efpd, but considering the filtrations $\{J_n = I^{(n)}\}$, $\{J_n' = I^{(2n)}\}$ and $\{J_n'' = I^{(2n+1)}\}$, all of which are equivalent to $\{ I^n \}$, we see that $\lim_{n \rightarrow \infty} \pd(R/J_n)$ does not exist, while $\lim_{n \rightarrow \infty} \pd(R/J_n')$ and $\lim_{n \rightarrow \infty} \pd(R/J_n'')$ exist but converge to different values. These examples show that :
    \begin{enumerate}
        \item $I$ has efpd $\not\Rightarrow \quad \lim_{n \rightarrow \infty} \pd(R/J_n)$ exists for any filtration equivalent to $\{ I^n \}$.
        \item $I$ has efpd $\not\Rightarrow \quad \lim_{n \rightarrow \infty} \pd(R/J_n) = \lim_{n \rightarrow \infty} \pd(R/J_n')$ for filtrations $\{ J_n \}$ and $\{ J_n' \}$ equivalent to $\{ I^n \}$ such that both limits exist.
    \end{enumerate}
The following recent result also shows that the power filtration may not yield information about whether an ideal has efpd.
\begin{lemma*}\label{tony}\cite[Lemma 3.1]{ganesh-tony}
    Let $(R,\m)$ be a Cohen-Macaulay local ring of dimension one. Let $I$ be an $\m$-primary and non-principal ideal, then $\pd(R/I^n) = \infty$ for all $n \geq 1$.
\end{lemma*}
Since $R$ is a Cohen-Macaulay local ring of dimension one, there exists an $\m$-primary ideal $J$ generated by an $R$-regular element. Therefore, $\{J^n\}_{n \gg 0}$ gives a filtration of $I$ equivalent to the $\{I^n\}$-filtration, consisting of projective modules, and hence $I$ has efpd.

To summarize, it appears that characterizing when an ideal has efpd through a filtration-independent numerical invariant based on their eventual behaviour is not possible.

\section*{Acknowledgments}
We are thankful to Srikanth Iyengar who suggested looking at the Artin-Rees property and vanishing of the maps on Tor modules as in Proposition \ref{andre-homology}, in the context of Section \ref{section main thm} and also for suggesting various improvements after kindly going through a first draft of this article. The first suggestion in particular  greatly decreased the length of the original proof of \ref{chainmapexists} and opened new directions for us to look into. We are thankful to ICTS, and particularly the program ICTS/dta2023/05, which gave us the freedom to work on this article without distractions, and the opportunity and space to meet and discuss with Srikanth. We are also thankful to V. Srinivas who suggested Example \ref{efpd example}(\ref{Dutta Hochster Mclaughlin eg}). 

The second named author would like to thank J.K.Verma for useful discussions related to Example \ref{equivalent filtration eg}(ii) and Billy Sanders for many conversations related to the contents of this article including the role of the Tate resolution from a decade ago.

The first named author would like to acknowledge the support from the Prime Minister's Research Fellowship (PMRF) scheme for carrying out this research work.

\end{document}